\documentclass[reqno]{amsart}
\usepackage{amsthm,amssymb}
\usepackage{amsmath,mathtools}
\usepackage{graphicx}
\usepackage{tikz}
\usepackage[normalem]{ulem}
\usepackage{cancel}
\usepackage{xcolor}

\mathtoolsset{showonlyrefs}

\newtheorem{theorem}{Theorem}[section]
\newtheorem{lemma}[theorem]{Lemma}
\newtheorem{proposition}[theorem]{Proposition}
\newtheorem{corollary}[theorem]{Corollary}
\newtheorem{remark}[theorem]{Remark}
\newtheorem{example}[theorem]{Example}

\newcommand{\rd}{{\rm d}}
\newcommand{\e}{{\rm e}}

\newcommand{\N}{{\mathbb N}}
\newcommand{\R}{{\mathbb R}}
\newcommand{\C}{{\mathbb C}}
\newcommand{\Z}{{\mathbb Z}}

\newcommand{\dist}{\mathrm{dist}}
\newcommand\re{\mathrm{Re}\,}
\newcommand\im{\mathrm{Im}\,}
\newcommand\I{\mathrm{i}}

\newcommand\E{\mathcal{E}}

\DeclareMathOperator{\supp}{supp}
\DeclareMathOperator{\diam}{diam}

\usepackage{ifthen} 

\newboolean{showrevisions}
\setboolean{showrevisions}{true}

\ifthenelse{\boolean{showrevisions}}
{
    
    \newcommand{\red}[1]{\textcolor{red}{\sout{#1}}}

}
{
    
    \newcommand\red[1]{}
}

\begin{document}

\title[Effective upper bounds in potential scattering]{Effective upper bounds on the number of resonances in potential scattering}
 \author[J.-C.\ Cuenin]{Jean-Claude Cuenin}
 \address[J.-C.\ Cuenin]{Department of Mathematical Sciences, Loughborough University, Loughborough,
 Leicestershire, LE11 3TU United Kingdom}
 \email{J.Cuenin@lboro.ac.uk}

\begin{abstract}
We prove upper bounds on the number of resonances and eigenvalues of Schrödinger operators $-\Delta+V$ with complex-valued potentials, where $d\geq 3$ is odd. The novel feature of our upper bounds is that they are \emph{effective}, in the sense that they only depend on an exponentially weighted norm of V. Our main focus is on potentials in the Lorentz space $L^{(d+1)/2,1/2}$, but we also obtain new results for compactly supported or pointwise decaying potentials. The main technical innovation, possibly of independent interest, are singular value estimates for Fourier-extension type operators. The obtained upper bounds not only recover several known results in a unified way, they also provide new bounds for potentials which are not amenable to previous methods. 
\end{abstract}

\maketitle


\section{Introduction}
\subsection{Counting resonances}
Assume that $d\geq 3$ is odd and that $V\in L^{\infty}_{\rm comp}(\R^d)$ is a bounded, compactly supported potential, possibly \emph{complex-valued}. The resolvent 
\begin{align*}
R_V(\lambda):=(-\Delta+V-\lambda^2)^{-1}:L^2(\R^d)\to L^2(\R^d),\quad \im\lambda\gg 1,
\end{align*}
extends to a meromorphic family (see \cite[Thm. 3.8]{MR3969938})
\begin{align*}
R_V(\lambda):L^2_{\rm comp}(\R^d)\to L^2_{\rm loc}(\R^d),\quad \lambda\in\C. 
\end{align*}
Scattering resonances are defined as the poles of this meromorphic continuation. Eigenvalues $z=\lambda^2$ correspond to resonances in the upper half plane.
Let $n_V (r)$ denote the number of resonances (counted with multiplicity) with absolute value at most $r$,
\begin{align*}
n_V (r):=\#\{\mbox{resonances }\lambda:\,|\lambda|\leq r\}.
\end{align*} 
The first polynomial bound $n_V(r)\leq C_Vr^{d+1}$ was proved by Melrose \cite{MR724031}.
Zworski \cite{MR1016891} proved the sharp upper bound
\begin{align}\label{Zworskis upper bound}
n_V(r)\leq C_Vr^d,\quad r\geq 1,
\end{align}
see also \cite[Thm. 3.27]{MR3969938} for a textbook presentation which uses a substantial simplification of the argument due to Vodev \cite{MR1163673}. Christiansen and Hislop \cite{MR2189242,MR2648080} proved that \eqref{Zworskis upper bound} is optimal for (Baire) generic complex or real-valued potentials in the sense that $\limsup_{r\to\infty}\log n(r)/\log r=d$.
An example of Christiansen shows that there are complex-valued potentials with no resonances \cite{MR2225694}. The question of whether \eqref{Zworskis upper bound} is optimal for arbitrary real-valued $V\in L^{\infty}_{\rm comp}(\R^d)$ is still open. 
Asymptotics for $n_V(r)$ are known in $d=1$ \cite{MR899652,MR1456597,MR95702,MR1802901} and for certain radial potentials \cite{MR987299,MR2190165}, or generic (in the sense of pluripotential theory) non-radial potentials supported in a ball in $d\geq 3$ \cite{MR3162489}. 
For potentials decaying like $\exp(-|x|^{1+\epsilon})$, Sá Barreto and Zworski \cite{MR1355631} proved
\begin{align}\label{Sa Barreto and Zworski}
n_V(r)\leq C_Vr^{d(1+1/\epsilon)},\quad r\geq 1.
\end{align}
For super-exponentially decaying potentials, Froese \cite{MR1629819} established an upper bound on $n_V(r)$ in terms of the growth of the Fourier transform of $V$. Returning to the compactly supported case, the estimate \eqref{Zworskis upper bound} admits far reaching generalizations to operators other than $-\Delta+V$, e.g.\ metric perturbations, obstacle scattering and scattering by finite-volume surfaces. The appropriate abstract framework is called \emph{black box scattering}. We will not discuss this further but refer to the recent book \cite{MR3969938} and survey article \cite{MR3625851} for more details and references.
We also consider the \emph{semiclassical} Schrödinger operator $-h^2\Delta+V$. It is clear that $n_V(r,h)$, the number of resonances $\leq r$, is equal to $n_{V/h^2}(r/h)$. The semiclassical analogue of~\eqref{Zworskis upper bound} (see \cite[Thm. 4.13]{MR3969938}) is
\begin{align}\label{semiclassical upper bound}
n_V(r,h)\leq C_Vh^{-d}r^d,\quad r\geq 1,\quad h\in (0,1].
\end{align}

\subsection{Effective bounds}
The previously discussed upper bounds are meaningful for a \emph{fixed} potential. Our main aim is to prove \emph{effective} upper bounds that makes the dependence of $C_V$ on $V$ \emph{explicit}. Korotyaev \cite[Theorem 1.1]{MR3574651} proved an effective upper bound in $d=1$, a simplified version of which (without explicit constants) being
\begin{align}\label{Korotyaev}
n_V(r)\lesssim ra+\ln (2+r)+\frac{Q}{1+r},\quad r>0,
\end{align}
provided $\supp(V)\subset [-a,a]$ and $Q:=\int_{\R}(1+|x|)|V(x)|\rd x<\infty$.
Our contribution is a far-reaching generalization of \eqref{Korotyaev} to higher dimensions
and to larger potential classes (Lorentz space, compactly supported or pointwise decaying).
For the sake of exposition, we temporarily continue to assume that $V$ is bounded and compactly supported (so that resonances are defined as above), but our estimates will be uniform in the respective potential class. We postpone the definition of resonances for non compactly supported potentials to Section \ref{sect. Meromorphic continuation}. The appearance of exponentially weighted norms is natural in view of the structure of the Schwartz kernel of the free resolvent (see e.g.\ \cite[3.1.16]{MR3969938}). For the definition of the Lorentz norm $\|V\|_{(d+1)/2,1/2}$, the reader is referred to Section 5.1
\begin{theorem}\label{thm. main intro}
Suppose $d\geq 3$ is odd and that $V$ is bounded and compactly supported with $\supp V\subset B(0,R)$. For every $\epsilon>0$ there exists a constant $C_1=C_1(\epsilon)$  (independent of $V$ and $R$), such that every resonance $\lambda$ satisfies
    \begin{align}
    |\lambda|&\leq C_1\|\e^{2(1+\epsilon)(\im\lambda)_-|\cdot|}V\|_{(d+1)/2,1/2}^{(d+1)/2},\label{main thm resonance free region Lp}\\
    |\lambda|&\leq C_1 R\langle\lambda R\rangle^{\epsilon} \e^{2(1+\epsilon)(\im\lambda)_-R}\|V\|_{\infty},\label{main thm resonance free region cpt}\\
    |\lambda|&\leq C_1 \langle\lambda \rangle^{\epsilon}\|\e^{2(1+\epsilon)(\im\lambda)_-|\cdot|}\langle\cdot\rangle^{1+\epsilon}V\|_{\infty},\label{main thm resonance free region ptw}
\end{align}
where $(\im\lambda)_-:=\max(0,-\im\lambda)$ and $\langle \lambda\rangle:=2+|\lambda|$.
Moreover, there exists an absolute constant $C_2\geq 2$ such that the number of resonances $|\lambda|\leq r$ is bounded by
\begin{align}
n_V(r)&\leq C_1 [\ln(C_1C_2r^{-2/(d+1)}\|\e^{2(1+\epsilon)r|\cdot|}V\|_{(d+1)/2,1/2})]^d\mbox{ if } r\geq C_2 \|V\|_{(d+1)/2}^{(d+1)/2},\label{main thm n_V Lp}\\
n_V(r)&\leq C_1 [\ln(C_1C_2 Rr^{-1}\langle r R\rangle^{\epsilon} \e^{2(1+\epsilon)rR}\|V\|_{\infty})]^d\mbox{ if } r\geq C_2 \|V\|_{\infty}^{1/2},\label{main thm n_V cpt}\\
n_V(r)&\leq C_1[\ln(C_1C_2\langle r \rangle^{\epsilon}r^{-1}\|\e^{2(1+\epsilon)r|\cdot|}\langle\cdot\rangle^{1+\epsilon}V\|_{\infty})]^d\!\mbox{ if } r\geq C_2 \|\langle\cdot\rangle^{d+\epsilon}V\|_{\infty}^{1/2}\label{main thm n_V ptw}.
\end{align}
 \end{theorem}
A more refined version of Theorem \ref{thm. main intro} is stated in Section \ref{sect. proof main thm. intro} (Theorem \ref{thm. main detailed}). 
The refinement has several aspects, among others:
\begin{enumerate}
    \item The dependence of $C_1$ on $\epsilon$ is explicit;
    \item The upper bound is valid for all $r>0$ (at the expense of having additional terms on the right hand side). We reiterate that our bounds are \emph{uniform}, not just asymptotic for large $r$. For this reason, we include terms like $r^{-1}$ in \eqref{main thm n_V cpt}, which would be negligible in the limit $r\to\infty$. They also play a role in the semiclassical bound and ensure that the argument of the logarithms in \eqref{main thm n_V Lp}, \eqref{main thm n_V cpt} are dimensionless, thus respecting dilation symmetry (scaling). The symmetry is broken in \eqref{main thm n_V ptw} since we have fixed a scale by using $\langle x\rangle$ in the norm. Theorem 7.1 will involve an arbitrary scale $R>0$, which restores the symmetry and makes the similarity to the compactly supported case more obvious.
    \item The upper bound holds for the number of resonances in a larger disk $D(\lambda_0,|\lambda_0|+r)$, where $\lambda_0\in\I\R_+$ (linearly) depends on a parameter $A$ that may be tuned. In particular, by choosing $A\gg r\gg 1$, one gets a bound on the number of resonances close to the real axis.
    \item There are additional parameters that may be tuned to change the exponential weights.
\end{enumerate}

\begin{remark}\label{remark literature results}
 The bounds \eqref{main thm n_V Lp}, \eqref{main thm n_V cpt}, \eqref{main thm n_V ptw} correctly reproduce the asymptotic results discussed before. More precisely, straightforward calculations reveal the following (see Appendix \ref{appendix examples} for proofs):
\begin{itemize}
    \item[(i)] 
\eqref{main thm n_V Lp}, \eqref{main thm n_V cpt}, \eqref{main thm n_V ptw} all imply Zworski's bound \eqref{Zworskis upper bound}. 
\item[(ii)] \eqref{main thm n_V Lp}, \eqref{main thm n_V ptw} both imply the bound \eqref{Sa Barreto and Zworski} of Sá Barreto and Zworski. 
\item[(iii)] \eqref{main thm n_V cpt}, \eqref{main thm n_V ptw} both imply the semiclassical bound \eqref{semiclassical upper bound}.
\end{itemize}
\end{remark}

The theorem also yields genuinely new bounds that cannot be obtained from any previous results in the literature.
\begin{example}\label{example sparse}
Assume $V=\sum_{j=1}^{\infty} H_j
\mathbf{1}_{\Omega_j}$, where $H_j\in \C$ and $\Omega_j\subset\R^d$ are mutually disjoint bounded measurable sets. Assume that $L_j:=\dist(\Omega_j,\bigcup_i\Omega_i\setminus\Omega_j)$ is increasing, $\sum_j\exp(-\eta L_j)<\infty$ for every $\eta>0$, and $\lim_{j\to\infty}L_j^{-1}\diam(\Omega_j)=0$. Potentials of this type were called \emph{sparse} in \cite{MR4426735}. For simplicity, assume that $|H_j|=1$, $L_j=j$ and $\Omega_j$ are balls with $|\Omega_j|\exp(Mj)<\infty$ for all $M>0$. In particular, $V$ need not decay at infinity. However, \eqref{main thm n_V Lp} implies that $n_V(r)$ is finite for all $r>0$ (see Section \ref{appendix examples} for details).
\end{example}

\subsection{Meromorphic continuation}\label{sect. Meromorphic continuation}
    By a simple density argument, the result of Theorem \ref{thm. main intro} can be extended to non compactly supported potentials as long as the respective norms are finite. In order to define resonances in this case we need the following results (see Section \ref{sect. singular value estimates} for the proofs).
\begin{proposition}\label{prop. meromorphic Lp}
Let $d\geq 3$ be odd, $\gamma>0$. If $V\in \e^{-2\gamma|\cdot|}L^{(d+1)/2,1/2}(\R^d)$, then $R_V$ has a meromorphic continuation to
$\im\lambda>-\gamma$.
\end{proposition}

\begin{proposition}\label{prop. meromorphic ptw.}
Let $d\geq 3$ be odd, $\gamma,\epsilon>0$. If $V\in \e^{-2\gamma|\cdot|}\langle \cdot\rangle^{-1-\epsilon}L^{\infty}(\R^d)$, then $R_V$ has a meromorphic continuation to
$\im\lambda>-\gamma$.
\end{proposition}

\subsection{Resonances in the upper half plane}
The techniques used to prove resonance bounds are close to those used to prove \emph{eigenvalue bounds} for Schrödinger operators with \emph{complex potentials}. We refer for instance to \cite{MR4426735} for references on this subject.
Eigenvalues $z=\lambda^2$ are of course just resonances in the upper half-plane $\im\lambda\geq 0$. If $V$ is real-valued, then (e.g. by the Cwikel--Lieb--Rozenbljum bound if $d\geq 3$) there are only finitely many negative eigenvalues. If $V$ is complex-valued, Frank \cite{MR2820160} proved that all eigenvalues are contained in the disk 
\begin{align}\label{Frank's bounds}
|z|^{p-d/2}\lesssim \|V\|^{p}_{p}
\end{align}
for $d\geq 2$ and $d/2<p\leq (d+1)/2$. Bögli and the author \cite{bogli2021counterexample} produced a counterexample that shows \eqref{Frank's bounds} is false for $p>(d+1)/2$, thus disproving a conjecture of Laptev and Safronov \cite{MR2540070}. A consequence of this counterexample is that there exist potentials with arbitrary small $L^{p}$ ($p>(d+1)/2$) and $L^{\infty}$ norm whose eigenvalues accumulate to every point of the essential spectrum $[0,\infty)$. Compared to bounds on single eigenvalues, much less is known about the distribution and especially about the number of eigenvalues, i.e.\ the quantity
\begin{align*}
n_V^+:=\#\{\mbox{resonances }\lambda:\,\im\lambda\geq 0\}.
\end{align*} 
Frank, Laptev and Safronov \cite{MR3556444} proved that, when $d\geq 3$ is odd, then
\begin{align}\label{n_V+ FLS original bound intro}
n_V^+\lesssim \gamma^{-2}\|\e^{2\gamma|\cdot|}V\|_{(d+1)/2}^{d+1}
\end{align}
for any $\gamma>0$ (with implicit constant independent of $\gamma$).
The techniques developed here can be used to obtain an alternative (and sometimes sharper) bound. To set a benchmark, we test \eqref{n_V+ FLS original bound intro} on bounded potentials supported in $B(0,R)$, where $R$ is large and $\|V\|_{\infty}\leq 1$. Then \eqref{n_V+ FLS original bound intro} implies $n_V^+\lesssim R^{2d+2}$. Our method yields the improved bound $n_V^+\lesssim R^{2d}$. Indeed, by \cite{AADavies} (or by \eqref{cpt. singular values Ancona} with $k=1$ and the Birman--Schwinger principle), all resonances in the upper half plane (eigenvalues) must lie in a half-disk $|\lambda|\leq C_0 R$, $\im\lambda\geq 0$. Hence, choosing $r=C_0R$ in Theorem \ref{thm. main intro}, it follows from \eqref{main thm n_V cpt}
that $n_V^+\lesssim R^{2d}$ for all $R>\max(1,C_2/C_0)$.

\subsection{Overview of the proof}\label{sect. overview proof}
The proof of \eqref{Zworskis upper bound} starts by identifying resonances with a subset of the zeros of a certain determinant, see \cite{MR3969938} for a textbook presentation. The determinant there is
\begin{align}\label{determinant used in DZ}
\det(I-(VR_0(\lambda)\rho)^{d+1}),\quad\lambda\in\C,
\end{align}
where $\rho\in C_c^{\infty}(\R^d)$ is equal to $1$ on $\supp V$. Recall that, if $K$ is a trace class operator, the Fredholm determinant of $I-K$ is defined by
\begin{align*}
    \det(I-K):=\prod_{k\in\N}(1-\lambda_k(K)),
\end{align*}
where $\lambda_j(K)$ are the eigenvalues of $K$, repeated according to multiplicities.
The power $p=d+1$ in \eqref{determinant used in DZ} is convenient, but any integer $\alpha>d/2$ could be chosen. We will work with a different Fredholm determinant, 
\begin{align}\label{ptw. H(lambda) def.}
H_{\alpha}(\lambda):=\det(I-(-BS(\lambda))^{\alpha}),\quad\lambda\in\C,
\end{align}
where $\alpha>d/2$ is an integer,
defined in terms of the Birman-Schwinger operator 
\begin{align}\label{Birman--Schwinger operator}
BS(\lambda):=|V|^{1/2}R_0(\lambda)V^{1/2}.
\end{align}
Then the poles of the resolvent $R_V$ are among the zeros of $H_{\alpha}$ (this follows from the resolvent identity \eqref{resolvent identity}). Let $n_{H_{\alpha}}(r)$ be the number of zeros (counting multiplicities) of $H_{\alpha}$ in the disk $D(0,r)$.
The definition of the multiplicity $m_R(\lambda)$ of a resonance $\lambda$ is subtle, and we refer the reader to \cite[Sect. 3.2]{MR3969938}. However, for the purpose of proving upper bounds, it suffices to know that $m_R(\lambda)$ is less or equal than the multiplicity of $\lambda$ as a zero of $H_{\alpha}$, see \cite[Thm. 3.26]{MR3969938} (The proof there works for any integer $\alpha>d/2)$. In particular, this implies that $n_V(r)\leq n_{H_{\alpha}}(r)$.

In order to ensure that $H_{\alpha}$ is an entire function we need to restrict to bounded, compactly supported potentials; in the general case, the estimates on $BS(\lambda)$ do not allow us to exclude a singularity at the origin (see e.g.\ \eqref{Frank--Laptev--Safronov}). 
Froese \cite{MR1629819} also used a Birman--Schwinger argument, but he considers a different determinant (the determinant of the scattering matrix). Our definition is closer to \cite{MR3556444} and other works concerned with the distribution of \emph{eigenvalues} for Schrödinger operators with complex potentials, e.g.\ \cite{BorichevEtAl2009,MR2559715,MR3717979, MR3730931,MR2540070}. However, a major difference to these works is that we do not use regularized determinants. Zworski \cite{MR1016891} observed that, except when $d=3$, these grow too fast as $\im\lambda\to -\infty$ (also see the discussion before \cite[Thm. 3.27]{MR3969938}), and thus they become less useful for the purpose of counting resonances (in the lower half plane).

It will be convenient to consider the regularized counting functions
\begin{align}\label{N_V(r)}
N_V(r):=\int_0^r\frac{n_V(t)}{t}\rd t,\quad N_{H_{\alpha}}(r):=\int_0^r\frac{n_{H_{\alpha}}(t)}{t}\rd t,
\end{align}
for which we also have $N_V(r)\leq N_{H_{\alpha}}(r)$.
Moreover, for any $s>1$, we have that $n_V(r)\leq (\ln s)^{-1}N_V(sr)$ (see e.g.\ \cite{MR1456597}). It will therefore be sufficient to prove upper bounds on $N_{H_{\alpha}}(sr)$ for all $r>0$ and some fixed $s>1$. This is achieved by means of Jensen's formula,
\begin{align}\label{Jensen's formula}
    N_{H_{\alpha}}(sr)=\frac{1}{2\pi}\int_0^{2\pi}\ln|H_{\alpha}(sr\e^{\I\theta})|\rd \theta-\ln|H_{\alpha}(0)|.
\end{align}
Using Weyl's inequality between eigenvalues and singular values one can show that (see Lemma \ref{lemma Weyl type first term})
\begin{align}\label{intro first term in Jensen}
    \ln|H_{\alpha}(\lambda)|\lesssim  \sum_{k\in\N}\ln(1+[s_k(BS(\lambda))]^{\alpha}),
\end{align}
which gives control over the first term in \eqref{Jensen's formula}, provided we can prove good upper bounds on the singular values $s_k(BS(\lambda))$.
However, the second term is a nuisance since the upper bounds for the quantity $-\ln|H_{\alpha}(\lambda)|$ will blow up at $\lambda=0$. To circumvent this problem we will consider a different function 
\begin{align}\label{def. F_alpha}
F_{\alpha}(k):=\frac{H_{\alpha}(\lambda_0+k)}{H_{\alpha}(\lambda_0)},\quad k\in\C,
\end{align}
that is normalized at zero, i.e.\ $F_{\alpha}(0)=1$, and that counts resonances in a larger disk $D(\lambda_0,|\lambda_0|+sr)$, with $\lambda_0\in\I\R_+$ (see Figure 1). Choosing $|\lambda_0|$ sufficiently large, one can show that (see Lemma \ref{lemma Weyl type second term})
\begin{align}\label{intro second term in Jensen}
   -\ln|H_{\alpha}(\lambda_0)|\lesssim \sum_{k\in\N}[s_k(BS(\lambda_0))]^{\alpha}.
\end{align}
Applying Jensen's formula to $F_{\alpha}$ then yields 
\begin{align}\label{Overview upper bound n_V}
    n_V(r)\ln s\lesssim  \max_{\lambda\in\partial D(\lambda_0,|\lambda_0|+sr)} \sum_{k\in\N}\ln(1+[s_k(BS(\lambda))]^{\alpha})+
   \sum_{k\in\N}[s_k(BS(\lambda_0))]^{\alpha}.
\end{align}
The maximum will be estimated separately for the part of the boundary in the upper and in the lower half plane. In the upper half plane, we will prove polynomial bounds of the form
\begin{align}\label{Overview model bound s_k upper}
 s_k(BS(\lambda))\lesssim M_+(\lambda)k^{-1/\beta_+},\quad \im\lambda\geq 0,
\end{align}
where $M_+(\lambda)$ is one of the norms of $V$ appearing in Theorem \ref{thm. main intro} and $\beta_+>0$.
The singular value estimates for $\im\lambda<0$ will be based on \emph{Stone's formula} (see e.g.\ \cite[(3.1.19)]{MR3969938}) in the form
\begin{align}\label{BS Stone formula}
BS(\lambda)-BS(-\lambda)=a_d\lambda^{d-2}|V|^{1/2}\mathcal{E}(\lambda)\mathcal{E}(\overline{\lambda})^*V^{1/2},\quad \im \lambda<0,
\end{align}
where $\mathcal{E}(\lambda)f(x):=\mathcal{F}(f\rd S)(-\lambda x)$ denotes the (analytically continued) Fourier extension operator (see Section \ref{sect. Fourier extension operator}). We will prove \emph{exponential} bounds of the form
\begin{align}\label{Overview model bound s_k lower}
    s_k(|V|^{1/2}\mathcal{E}(\lambda)\mathcal{E}(\overline{\lambda})^*|V|^{1/2})\lesssim M_-(\lambda) \exp(-c\epsilon k^{1/\beta_-}),\quad \lambda\in\C,
\end{align}
where $M_-(\lambda)$ is of the same type as $M_+(\lambda)$ and $c,\beta_->0$. For instance, an admissible choice would be 
\begin{align}
M_+(\lambda)&=|\lambda|^{-\frac{2}{d+1}}\|V\|_{(d+1)/2,1/2}^{(d+1)/2},\\
M_-(\lambda)&=|\lambda|^{-\frac{d(d-1)}{d+1}}\|\e^{2(1+\epsilon)|\im\lambda||\cdot|}V\|_{(d+1)/2,1/2}^{(d+1)/2} \label{M_- example}
\end{align}
and $\beta_+=d+1$, $\beta_-=d-1$ (see Proposition \ref{prop. sk(BS) Lp}).
Using \eqref{Overview model bound s_k upper}, \eqref{BS Stone formula}, \eqref{Overview model bound s_k lower}, we can estimate \eqref{Overview upper bound n_V}, leading to an upper bound 
\begin{align}\label{n<nu+nl+n0}
    n_V(r)\lesssim (\ln s)^{-1}(\tilde{n}_+(sr)+\tilde{n}_-(sr)+\tilde{n}_0),
\end{align}
where $\tilde{n}_+(sr),\tilde{n}_-(sr)$ denote the contributions of the first term in \eqref{Overview upper bound n_V}, corresponding to the upper and lower half plane, respectively, and $\tilde{n}_0$ denotes the contribution of the second term in \eqref{Overview upper bound n_V}. Of course, these quantities also depend on $V$, but only through one of its norms.
For large $r$, the dominant term is $n_-(r)$, and this leads to the bounds \eqref{main thm n_V Lp}, \eqref{main thm n_V cpt}, \eqref{main thm n_V ptw} stated in Theorem \ref{thm. main intro}. The resonance-free regions can be determined by the Birman--Schwinger principle, i.e.\ the fact that $\lambda^2$ is an eigenvalue of $-\Delta+V$ if and only if $-1$ is an eigenvalue of $BS(\lambda)$. This implies that $\|BS(\lambda)\|\geq 1$ for any resonance, which yields \eqref{main thm resonance free region Lp}, \eqref{main thm resonance free region cpt}, \eqref{main thm resonance free region ptw}. Since the operator norm equals the first singular value, these will therefor also follow from the singular value bounds.

\subsection{Technical novelties}
The main technical novelty in this work compared to the textbook result \cite{MR3969938} is the uniform control
of the singular values \eqref{Overview model bound s_k upper}, \eqref{Overview model bound s_k lower} in terms of the potential.
By keeping track of the constants in the presentation of \cite{MR3969938} one could in principle also obtain effective estimates, in terms of $\|V\|_{\infty}$ and the size of the support of $V$. However, compared to \eqref{main thm n_V cpt} (the compactly supported case), this bound would involve $R^d\|V\|_{\infty}$ as opposed to $R\|V\|_{\infty}$. This is because instead of the trivial bound for the Fourier extension operator (taking absolute values inside the integral) we use a new energy estimate (Lemma \ref{lemma Agmon--Hörmander complex}). The latter can be viewed as a generalization of the classical Agmon--Hörmander bound \cite{MR0466902} to ``complex energies".
Similar results were obtained by Burq \cite{MR1876933} and Gannot \cite{MR3393682}, but our estimate is scale-invariant and is nearly optimal in the exponential weight.
The singular value estimates for compactly supported potentials can be summed dyadically to yield the results for pointwise decaying potentials.

The main advancement in this article is that we allow potentials with only average decay (the Lorentz case). To the best of our knowledge, our bounds on the number of resonances and on resonance-free regions are the first that involve only a (weighted) Lorentz norm. The corresponding singular value estimates are related to a result of Frank--Laptev--Safronov \cite{MR3556444}, who proved
\begin{align}\label{Frank--Laptev--Safronov}
\|BS(\lambda)\|_{\mathfrak{S}^{d+1}}\lesssim |\lambda|^{-\frac{2}{d+1}}\|\e^{2C_d(\im\lambda)_-|\cdot|}V\|_{(d+1)/2},\quad\lambda\in\C,
\end{align}
for some $C_d>1$. The case $\im\lambda\geq 0$ was established by Frank--Sabin \cite{MR3730931}.
In view of the exponential growth of the free resolvent, the weight cannot be smaller than $\e^{2(\im\lambda)_-|\cdot|}$.
By Stone's formula, \eqref{Frank--Laptev--Safronov} implies a polynomial bound on the singular values of the operator on the left of \eqref{Overview model bound s_k lower} (involving the Fourier extension operator $\mathcal{E}(\lambda)$). If $\lambda$ were real, then this would be a $TT^*$ operator, and its singular values would coincide with those of $T^*T$. The latter is an operator on the unit sphere, and one expects its singular values to decay exponentially (if $V$ has sufficient decay). If $\lambda$ is complex, then $\mathcal{E}(\lambda)\mathcal{E}(\overline{\lambda})^*$ is not a $TT^*$ operator, but it still has a translation-invariant kernel. However, to recover the exponential decay of singular values, one has to look at $\mathcal{E}(\lambda)\mathcal{E}(\lambda)^*$, which is not controlled by a translation-invariant kernel. Hence, the (by now) standard methods of proving Schatten norm estimates for such operators (either complex interpolation \cite{MR3730931} or the multilinear Hardy--Littlewood--Sobolev inequality \cite{MR3254332}) are not readily applicable. To overcome this difficulty, we adapt Bourgain's proof of the Stein--Tomas inequality \cite{MR1097257} to the trace ideal setting. It turns out that this method is robust enough to handle the case of complex $\lambda$. Moreover, it yields estimates that are stronger than \eqref{Frank--Laptev--Safronov} in two important ways:
\begin{enumerate}
\item We obtain exponential decay of singular values in \eqref{Overview model bound s_k lower} and
\item The weight $\e^{2(1+\epsilon)(\im\lambda)_-|\cdot|}$ in \eqref{M_- example} is essentially optimal (possibly up to an $\epsilon$ loss).
\end{enumerate}
However, our method only gives a restricted weak type inequality. This means two things: First, the Schatten space $\mathfrak{S}^{d+1}$ has to be replaced by the weak Schatten space $\mathfrak{S}^{d+1,w}$. Second, the $L^{(d+1)/2}$ of the potential has to be replaced by the $L^{(d+1)/2,1/2}$ Lorentz norm. 

In addition to the technical innovations discussed above, there are two additional new features worth highlighting. The first is a stationary phase estimate for certain oscillatory integrals with complex-valued phase functions. Although such integrals are covered e.g.\ in Hörmander's treatise of stationary phase \cite{MR1065993}, our phase functions are more special and allow for a more elementary proof (without reference to the Malgrange preparation theorem). The second new feature is a structure formula for $\mathcal{E}(\lambda)(I-\epsilon^2\Delta_S)$, where $\Delta_S$ the Laplace--Beltrami operator on the unit sphere $\mathbb{S}^{d-1}$. This is needed in order to reduce the exponential decay estimate \eqref{Overview model bound s_k lower} to Weyl's asymptotics. The trick is not new (see \cite{MR3969938} and references there) but the argument is much more delicate since we take cancellations of the phase into account. Again, if one only uses size estimates (taking absolute values inside the integral), this step becomes almost trivial.

\subsection{Outline of the paper}
In Section \ref{sect. Preliminaries}, we briefly review some standard results about singular values of compact operators, the particular resolvent identity used in this paper, and the Fourier extension operator $\mathcal{E}(\lambda)$. In Section \ref{sect. Agmon--Hörmander estimate}, we we prove a generalization of the Agmon--Hörmander bound for $\mathcal{E}(\lambda)$ to complex $\lambda$. In Section~\ref{sect. Stein--Tomas estimate}, we prove a similar generalization of the Stein--Tomas bound. Section \ref{sect. singular value estimates} contains the key estimates for the singular values of the Birman--Schwinger operator for the potential classes discussed in the introduction (Lorentz, compactly supported, poitnwise decaying).
In Section \ref{sect. Fredholm det.} we discuss Fredholm determinants and make the arguments in the introduction more precise. In Section \ref{sect. proof main thm. intro} we prove very detailed effective bounds for the number of resonances, Theorem \ref{thm. main detailed}, and we show that Theorem \ref{thm. main intro} follows from this. 
In Appendix A, we give the proof of a certain `structure formula' for commuting $\mathcal{E}(\lambda)$ with powers of the Laplacian on the unit sphere. In Appendix B, we provide additional details of the proof of Theorem 7.1. In Appendix C, we prove Remark \ref{remark literature results} and Example \ref{example sparse}.


\subsection*{Notation}
We write $A\lesssim B$ or $A=\mathcal{O}(1)B$ for two non-negative quantities $A,B$ to indicate
that there is a constant $C>0$ such that $A\leq C B$. To highlight the dependence of the constant on a parameter $\epsilon$, the notation $A\lesssim_{\epsilon}B$ is sometimes used. 
The dependence on fixed parameters like the dimension $d$ is usually omitted.
The notation $A\asymp B$ means $A\lesssim B\lesssim A$, and $A\ll B$ means that $A\leq cB$ for some sufficiently small constant $c>0$.
If $T:X\to Y$ is a bounded linear operator between two Hilbert spaces $X$ and $Y$, we denote its operator norm by $\|T\|_{X\to Y}$, or just by $\|T\|$ if the spaces are clear from the context. The indicator function of a set $\Omega\subset \R^d$ is denoted by $\mathbf{1}_{\Omega}$. We define $\langle x\rangle:=2+|x|$. The natural numbers are denoted by $\N=\{1,2,\ldots\}$. The upper/lower half plane is denoted by $\C^{\pm}=\{\lambda\in\C:\,\pm\im\lambda>0\}$. We also use the following abbreviations, especially in Sections 5 and 6,
\begin{align}
     c_{\delta}:=\delta^{-\frac{(d-1)^2}{d+1}},\quad
    \rho_{\delta}(x,\lambda):=\e^{-\sqrt{1+\delta}|\im\lambda||x|},\quad
        \rho_{\delta,\epsilon}(x,\lambda):=\e^{-(\sqrt{1+\delta}|\im\lambda|+\epsilon|\lambda|)|x|}.
\end{align}

\subsection*{Acknowledgments}
The author thanks Alexei Stepanenko and Konstantin Merz for useful discussions, and Maciej Zworski for valuable feedback on an earlier version of the article. He also thanks the anonymous reviewers for their careful reading of the manuscript and numerous comments, which significantly improved the paper. Support through the Engineering \& Physical Sciences Research Council (EP/X011488/1) is acknowledged.

\section{Preliminaries}\label{sect. Preliminaries}

\subsection{Singular values and Schatten spaces}
We denote by $\mathfrak{S}^p$ the Schatten class of order $p\in (0,\infty)$ over the Hilbert space $L^2(\R^d)$ and by $\|\cdot\|_{\mathfrak{S}^p}$ the corresponding Schatten norm
\begin{align*}
\|K\|_{\mathfrak{S}^p}:=\big(\sum_{k\in\N}s_k(K)^p\big)^{\frac{1}{p}},
\end{align*}
where $(s_k(K))_{k\in\N}$ denotes the sequence of singular values of the compact operator $K$, that is, the square roots of the eigenvalues of $K^*K$, in nonincreasing
order and repeated according to multiplicities (note that \cite{MR3969938} use a different convention, where the enumeration starts at $k=0$). The weak Schatten class $\mathfrak{S}^{p,\rm w}$ of order $p\in (0,\infty)$ is the set of all compact operators on $L^2(\R^d)$ for which the quasinorm $\|K\|'_{\mathfrak{S}^{p,\rm w}}:=\sup_{k\in\N}k^{\frac{1}{p}}s_k(K)$ is finite. For $p>1$, an equivalent \emph{norm} is given by
\begin{align*}
\|K\|_{\mathfrak{S}^{p,\rm w}}:=\sup_{N\in\N}N^{\frac{1}{p}-1}\sum_{k=1}^Ns_k(K),
\end{align*} 
see e.g.\ \cite[Ch. 11, Sect. 6]{MR1192782}. 
We will use the following standard facts for singular values repeatedly throughout the article (see e.g.\ \cite{MR2154153}). If $A,B$ are compact operators, then for $j,k\geq 0$,
\begin{align}
s_{j+k+1}(AB)&\leq s_{j+1}(A)s_{k+1}(B),\label{s_j(AB)}\\
s_{j+k+1}(A+B)&\leq s_{j+1}(A)+s_{k+1}(B)\label{s_j(A+B)}.
\end{align} 
If $A$ is compact and $B$ is bounded, then
\begin{align}
    s_k(AB),s_k(BA)\leq s_k(A)\|B\|.
\end{align}
Moreover, we have the elementary equality $s_k(K)=s_k(K^*)$.

\subsection{Resolvent identity}
If $V$ is bounded, then, by iterating the second resolvent identity,
\begin{align*}
R_V(\lambda)=R_0(\lambda)-R_0(\lambda)VR_{V}(\lambda), \quad\im\lambda\gg 1,
\end{align*}
one obtains
\begin{align}\label{resolvent identity}
R_V(\lambda)=R_0(\lambda)-R_0(\lambda)V^{1/2}(I+BS(\lambda))^{-1}|V|^{1/2}R_0(\lambda).
\end{align}
This formula is valid for $\im\lambda\gg 1$ since
$\|BS(\lambda)\|\leq |\im(\lambda^2)|^{-1}\|V\|_{\infty}$.
If $V$ is unbounded, then one can use \eqref{resolvent identity} as the \emph{definition} of $R_V(\lambda)$, $\im\lambda\gg 1$. This is a classical construction going back to Kato \cite{MR0190801}, see also  \cite{MR2201310} for abstract results in the nonselfadjoint setting.

\subsection{Fourier extension operator}\label{sect. Fourier extension operator}
Consider the (analytically continued) Fourier extension operator
\begin{equation}\label{Fourier extension operator}
\begin{split}
&\mathcal{E}(\lambda):L^1(\mathbb{S}^{d-1})\to L^{\infty}_{\rm loc}(\R^d),\\
&\mathcal{E}(\lambda)g(x)=\int_{S^{d-1}}\e^{\I\lambda x\cdot\xi}g(\xi)\rd S(\xi),\quad x\in\R^d,\quad \lambda\in\C,
\end{split}
\end{equation}
where $\rd S$ denotes induced Lebesgue measure on the unit sphere $\mathbb{S}^{d-1}$. It is immediate from the triangle inequality that $\mathcal{E}(\lambda)$ is a bounded operator. By Cauchy--Schwarz, it also follows that $\mathcal{E}(\lambda):L^2(\mathbb{S}^{d-1})\to L^{2}_{\rm loc}(\R^d)$ is bounded. For $\lambda\geq 0$, the following estimates hold for any $g\in C^{\infty}(\mathbb{S}^{d-1})$:
\begin{align}
\|\mathcal{E}(\lambda) g\|_{L^{\frac{2(d+1)}{d-1}}(\R^d)}&\lesssim \lambda^{-\frac{d(d-1)}{2(d+1)}}\|g\|_{L^2(\mathbb{S}^{d-1})},\label{ST for lambda>0}\\
\|\mathcal{E}(\lambda) g\|_{L^{2}(B(0,R))}&\lesssim R^{\frac{1}{2}}\lambda^{-\frac{d-1}{2}}\|g\|_{L^2(\mathbb{S}^{d-1})}.\label{AH for lambda>0}
\end{align}
These estimates follow by rescaling from $\lambda=1$, in which case \eqref{ST for lambda>0} is known as the \emph{Stein--Tomas} estimate \cite[Prop. IX.2.1]{MR1232192} and \eqref{AH for lambda>0} as the \emph{Agmon--Hörmander} estimate \cite[Thm. 2.1]{MR0466902}.

\section{Agmon--Hörmander estimate}\label{sect. Agmon--Hörmander estimate}
In this section we prove a generalization of the Agmon--Hörmander bound \eqref{AH for lambda>0}.
For later purposes it will be convenient to consider the slightly more general operators 
\begin{align}\label{def. Ea}
\mathcal{E}_a(\lambda)g(x):=\int_{\mathbb{S}^{d-1}}\e^{\I\lambda x\cdot\xi}a(x,\xi)g(\xi)\rd S(\xi),\quad x\in\R^d,\quad \lambda\in\C,
\end{align}
where $a(x,\xi)$ is bounded and smooth in the $\xi$-variable. We are looking for uniform estimates for $\E_a$ over a family of functions $a$ in the unit ball of $L^{\infty}(\R^d_x;C^N(\mathbb{S}_{\xi}^{d-1}))$:
\begin{align}\label{family a uniform bound}
\|a\|_{L^{\infty}(\R^d_x;C^N(\mathbb{S}_{\xi}^{d-1}))}\leq 1.
\end{align}
Here $N$ is a sufficiently large (depending only on the dimension) but fixed integer.
The following proposition is an analogue of the Agmon--Hörmander bound \eqref{AH for lambda>0} for $\im\lambda <0$ and is one of the main technical ingredients. It improves upon \cite[(2.5)]{MR1876933} and \cite[Lemma 1.1]{MR3393682}.

\begin{proposition}\label{lemma Agmon--Hörmander complex}
Let $d\geq 2$. $\lambda\in\C$. For each $R>0$, $\delta\in (0,1]$, we have 
\begin{align}\label{Agmon--Hörmander complex}
\|\rho_{\delta}\E_a(\lambda)g\|_{_{L^2(B(0,R))}}\lesssim  R^{\frac 12}|\delta\lambda|^{-\frac{d-1}{2}}\|g\|_{L^2(\mathbb{S}^{d-1})},
\end{align}
uniformly subject to \eqref{family a uniform bound}, where $\rho_{\delta}(x,\lambda):=\e^{-\sqrt{1+\delta}|\im\lambda||\cdot|}$.
\end{proposition}

For the proof we will need the following lemma;
for $\im\lambda=0$ it is an immediate consequence of Plancherel's theorem.

\begin{lemma}\label{prop. complex Plancherel}
Let $n\geq 1$, $\lambda\in\C$, $\epsilon>0$. If $f$ is supported in $B(0,\epsilon/2)\subset\R^n$, then
\begin{align}\label{complex Plancherel}
\|\e^{-\epsilon|\im\lambda||\cdot|}\widehat{f}(\lambda\,\cdot)\|_{L^2(\R^n)}\lesssim  |\lambda|^{-\frac{n}{2}}\|f\|_{L^2(\R^n)}.
\end{align} 
\end{lemma}

\begin{proof}
1. By scaling, we may assume without loss of generality that $\epsilon=1$. Indeed, if $f$ is supported in $B(0,\epsilon/2)$, then $g(\xi):=f(\epsilon\xi)$ is supported in $B(0,1/2)$, and if \eqref{complex Plancherel} holds with $g$ in place of $f$ and with $\epsilon=1$, then
\begin{align*}
    \epsilon^{-n}\|\e^{-|\im\lambda||\cdot|}\widehat{f}(\lambda/\epsilon\,\cdot)\|_{L^2(\R^n)}\lesssim  |\lambda|^{-\frac{n}{2}}\epsilon^{-\frac{n}{2}}\|f\|_{L^2(\R^n)}.
\end{align*}
Changing variables $\lambda=\epsilon\widetilde{\lambda}$ yields \eqref{complex Plancherel} wth $\widetilde{\lambda}$ in place of $\lambda$.
By the monotone convergence theorem it suffices to prove the estimate for $\epsilon=1$ over a ball $B(0,R)$ with arbitrary $R>0$. By a change of scale $x=Ry$, the inequality for $R=1$ and $\lambda$ replaced by $\lambda R$ yields the general case $R>0$.

2. Let $\chi_1,\chi_2$ be non-negative bump functions such that $\chi_1\geq \mathbf{1}_{B(0,1)}$, $\supp(\chi_1)\subset B(0,2)$ and $\chi_2\geq \mathbf{1}_{B(0,1/2)}$, $\supp(\chi_2)\subset B(0,3/4)$. Because of the support assumption on $f$, we then have 
\begin{align}\label{Appendix A FT f chi23}
\|\e^{-|\im\lambda||\cdot|}\widehat{f}(\lambda\,\cdot)\|_{L^2(B(0,1))}^2\leq 
\|\chi_1\e^{-|\im\lambda||\cdot|}\widehat{\chi_2 f}(\lambda\,\cdot)\|_{L^2(\R^n)}^2.
\end{align}
Since $\chi_2$ is smooth and supported in $B(0,3/4)$, it follows that $\widehat{\chi_2}$ is an entire function, satisfying
\begin{align}\label{Paley-Wiener}
    |\widehat{\chi_2}(z)|\leq C_N(1+|z|)^{-N}\e^{\frac{3}{4}|\im z|}
\end{align}
for all $z\in\C^n$ and all $N>0$, see e.g. \cite[7.3]{MR1065993}. Thus,
\begin{align}
    |\widehat{\chi_2 f}(\lambda x)|&=|\int_{\R^n}\widehat{\chi_2}(\lambda x-y)\widehat{f}(y)\rd y|\\
    &\leq C_N \e^{\frac{3}{4}|\im \lambda||x|}\left(\int (1+|y|+|(\im\lambda) x|)^{-N}\rd y\right)^{\frac{1}{2}} \|f\|_{L^2},
\end{align}
where we used Cauchy--Schwarz and Plancherel and changed variables $y\to y-(\re\lambda) x$. This yields
\begin{align}\label{dydx}
    \|\chi_1\e^{-|\im\lambda||\cdot|}\widehat{\chi_2 f}(\lambda\,\cdot)\|_{L^2}^2\leq C_N\|f\|_{L^2}^2 \int (1+|y|+|(\im\lambda) x|)^{-N}\chi_1(x)^2\rd y\rd x.
\end{align}
Since $\chi_1$ has compact support, this is bounded by $\|f\|_{L^2}^2$, which is the desired bound if $|\lambda|\leq 1$.

3. If $|\lambda|>1$, we distinguish between $|\im\lambda|\geq |\re\lambda|$ or $|\im\lambda|<|\re\lambda|$. In the first case, we change the order of integration in \eqref{dydx} (Fubini--Tonelli) and change variables $x\to |\im\lambda|^{-1}x$ to obtain an upper bound of $|\im\lambda|^{-n}\|f\|_{L^2}^2$, which is the desired bound in this case. If $|\im\lambda|<|\re\lambda|$, instead of using Cauchy--Schwarz in \eqref{Paley-Wiener}, we only estimate
\begin{align}
    |\widehat{\chi_2 f}(\lambda x)|
    \leq C_N \e^{\frac{3}{4}|\im \lambda||x|}\int (1+|(\re\lambda) x-y|)^{-N}|\widehat{f}(y)|\rd y
\end{align}
and use 
\begin{align}\label{Schwartz tails}
    (1+s)^{-N}\lesssim \mathbf{1}\{s\leq 1\}+\sum_{j=0}^{\infty}2^{-Nj}\mathbf{1}\{2^{j-1}\leq s\leq 2^j\}
\end{align}
for $s=|(\re\lambda) x-y|$. We use Cauchy--Schwarz for each term separately,
\begin{align}
&\int \mathbf{1}\{|(\re\lambda) x-y|\leq 1\}|\widehat{f}(y)|\rd y\lesssim \left(\int_{|(\re\lambda) x-y|\leq 1} |\widehat{f}(y)|^2\rd y\right)^{\frac{1}{2}},\\
 &\int \mathbf{1}\{2^{j-1}\leq |(\re\lambda) x-y|\leq 2^j\}|\widehat{f}(y)|\rd y\lesssim 2^{\frac{nj}{2}}\left(\int_{2^{j-1}\leq |(\re\lambda) x-y|\leq 2^j} |\widehat{f}(y)|^2\rd y\right)^{\frac{1}{2}}.
\end{align}
Then, for $N>n$, \eqref{Schwartz tails} implies
\begin{align}
   & \|\chi_1\e^{-|\im\lambda||\cdot|}\widehat{\chi_2 f}(\lambda\,\cdot)\|_{L^2}^2
   \lesssim_N \int\int_{|(\re\lambda) x-y|\leq 1} |\widehat{f}(y)|^2\rd y\rd x\\
    &+\int\left|\sum_{j=0}^{\infty}2^{\frac{nj}{2}-N}\left(\int_{2^{j-1}\leq |(\re\lambda) x-y|\leq 2^j} |\widehat{f}(y)|^2\rd y\right)^{\frac{1}{2}}\right|^2\rd x\\
    &\lesssim \int\int_{|(\re\lambda) x-y|\leq 1} |\widehat{f}(y)|^2\rd y\rd x+\sum_{j=0}^{\infty}2^{(n+1)j-2Nj}\int\int_{2^{j-1}\leq |(\re\lambda) x-y|\leq 2^j} |\widehat{f}(y)|^2\rd y\rd x,
\end{align}
where we used Cauchy--Schwarz for the sum in the last line. By Fubini--Tonelli, 
\begin{align}
    &\int\int_{|(\re\lambda) x-y|\leq 1} |\widehat{f}(y)|^2\rd y\rd x\lesssim|\re\lambda|^{-n}\int |\widehat{f}(y)|^2\rd y,\\
    &\int\int_{2^{j-1}\leq|(\re\lambda) x-y|\leq 2^j} |\widehat{f}(y)|^2\rd y\rd x\lesssim 2^{jn}|\re\lambda|^{-n}\int |\widehat{f}(y)|^2\rd y.
\end{align}
The resulting geometric series is summable for $2N>2n+1$, and Plancherel yields the desired bound.
\end{proof}

\begin{proof}[Proof of Proposition \ref{lemma Agmon--Hörmander complex}]
1.
We first give the proof for $\E(\lambda)$.
By a standard partition of unity argument, we may replace $\E(\lambda)$ by (with abuse of notation)
\begin{align}\label{replacement of E(lambda)}
\E_1(\lambda) g(x):=\int_{\R^{d-1}}\e^{\I\lambda(x'\cdot\xi'+x_1\psi(\xi'))}\chi(\xi')g(\xi')\rd\xi',
\end{align}
where $x=(x_1,x')$, $\xi=(\xi_1,\xi')$, $\psi(\xi')=\sqrt{1-|\xi'|^2}$ and $\chi$ is a smooth function that has support in a fixed small neighborhood of the origin. Since this neighborhood can be covered by approximately $\delta^{-\frac{d-1}{2}}$ finitely overlapping balls of diameter $C\delta^{\frac 12}$, it suffices to prove 
\begin{align}\label{Agmon--Hörmander complex suff. to prove}
\|\rho_{\delta}\E_1(\lambda)g\|_{_{L^2(B(0,R))}}\lesssim  R^{\frac 12}|\lambda|^{-\frac{d-1}{2}}\|g\|_{L^2(\mathbb{S}^{d-1})},
\end{align}
whenever $g$ is supported in $B(0,C\delta^{\frac 12})$.
We freeze the $x_1$ variable and denote the operator acting on the remaining variables $x'=(x_2,\ldots,x_d)$ by $\E_{x_1}(\lambda)$.
From the elementary estimate
$2|x_1||x'|\leq \delta x_1^2+\delta^{-1}|x'|^2$
it follows that
\begin{align*}
(|x_1|+|x'|)^2\leq (1+\delta)x_1^2+(1+\delta^{-1})|x'|^2.
\end{align*}
Rescaling $x_1\to\frac{x_1}{\sqrt{1+\delta}}$, $x'\to\frac{x'}{\sqrt{1+\delta^{-1}}}$ and multiplying the above inequality by $(1+\delta)$ we obtain
\begin{align}\label{x vs x' x1}
\sqrt{1+\delta}|(x_1,x')|\geq |x_1|+\epsilon |x'|\quad \mbox{with}\quad \epsilon:=\sqrt{(1+\delta)/(1+\delta^{-1})}.
\end{align}
Using this, one observes that 
\begin{align}\label{sqrt 1+delta ineq.}
|\e^{-\sqrt{1+\delta}|\im\lambda||(x_1,x')|}\E_{x_1}(\lambda) g(x')|\leq \e^{-\epsilon|\im\lambda||x'|}|\widehat{g_{\lambda,x_1}}(-\lambda x')|,
\end{align}
where $g_{\lambda,x_1}(\xi'):=\e^{-|\im\lambda||x_1|}\e^{\I\lambda x_1\psi(\xi')}\chi(\xi')g(\xi')$. Since $\frac{1}{4}\delta^{\frac{1}{2}}\leq \frac{1}{2}\epsilon$, Lemma \ref{prop. complex Plancherel} implies 
\begin{align*}
\|\e^{-\sqrt{1+\delta}|\im\lambda||(x_1,\cdot)|}\E_{x_1}(\lambda) g\|_{L^2(\R^{d-1})}\lesssim |\lambda|^{-\frac{d-1}{2}}\|g\|_{L^2(\R^{d-1})},
\end{align*}
whenever $\supp(\chi)\subset B(0,\frac{1}{4}\delta^{\frac{1}{2}})$, were we used that $|\e^{\I\lambda x_1\psi(\xi')}|\leq \e^{|\im\lambda||x_1|}$. 
Squaring and integrating over $|x_1|\leq R$ yields \eqref{Agmon--Hörmander complex suff. to prove}. 

2. To prove this for $\E_a$ we assume first that $a(x,\xi)$ factors as a product, $a(x,\xi)=b(x)c(\xi)$. Then the previous argument can be used without modification. In the general case, we expand $\chi a$ in a Fourier series,
\begin{align*}
\chi(\xi')a(x,(\psi(\xi'),\xi'))=\sum_{j'\in  \delta^{-\frac{1}{2}}\Z^{d-1}}a_{j'}(x)\e^{\I j'\cdot\xi'}.
\end{align*}
By a scaling argument, we may assume without loss of generality that $\delta=1$. Then, by \eqref{family a uniform bound}, we have $|a_{j'}(x)|\lesssim (1+|j'|)^{-N}$ for any $N>0$. The result for the general case then follows from the factored case and the triangle inequality.
\end{proof}

\begin{remark}
    From the proof of Proposition \ref{lemma Agmon--Hörmander complex} it transpires that if we replace the Euclidean norm $|\cdot|$ by the equivalent norm
    \begin{align*}
        |x|_0:=\max_{i=1,\ldots,d}(|x_i|+|x_i'|),
    \end{align*}
    where $x_i'\in\R^{d-1}$ is the vector $x$ with the $i$-th component omitted, then we have 
    \begin{align*}
        \|\e^{-|\im\lambda||\cdot|_0}\E_a(\lambda)g\|_{_{L^2(B(0,R))}}\lesssim  R^{\frac 12}|\lambda|^{-\frac{d-1}{2}}\|g\|_{L^2(\mathbb{S}^{d-1})}.
    \end{align*}
\end{remark}

\section{Stein--Tomas estimate}\label{sect. Stein--Tomas estimate}
Here we state a generalization of the Stein--Tomas bound \eqref{ST for lambda>0}. 
\begin{proposition}
Let $d\geq 2$, $\lambda\in\C$. Then for each $\delta\in (0,1]$,
\begin{align*}
\|\rho_{\delta}\mathcal{E}_a(\lambda)g\|_{L^{\frac{2(d+1)}{d-1}}(\R^d)}&\lesssim  \delta^{-\frac{(d-1)^2}{2(d+1)}} \lambda^{-\frac{d(d-1)}{2(d+1)}}\|g\|_{L^2(\mathbb{S}^{d-1})},
\end{align*}
uniformly subject to \eqref{family a uniform bound}.
\end{proposition}

We omit the proof since the result is a consequence of the stronger singular value estimate \eqref{eq. singular values 1 variant}. The key ingredient in the proof of the latter will be the following stationary phase estimate.

\begin{lemma}\label{corollary stat. phase}
Let $\lambda\in\C$. The Schwartz kernel $K_{\lambda}$ of $\mathcal{E}_a(\lambda)\mathcal{E}_a(\overline{\lambda})^*$ satisfies
\begin{align}\label{stat. phase I}
|K_{\lambda}(x-y)|\lesssim \e^{|\im\lambda||x-y|}(1+|\lambda||x-y|)^{-\frac{d-1}{2}}
\end{align} 
for $x,y\in\R^d$.
Moreover, the kernel $\widetilde{K}_{\lambda}$ of $\mathcal{E}_a(\lambda)\mathcal{E}_a(\lambda)^*$ satisfies
\begin{align}\label{stat. phase II}
|\widetilde{K}_{\lambda}(x,y)|\lesssim \e^{|\im\lambda||x+y|}(1+|\re\lambda||x-y|+|\im\lambda||x+y|)^{-\frac{d-1}{2}}.
\end{align} 
Both \eqref{stat. phase I}, \eqref{stat. phase II} are uniform over $a$ in the unit ball of $L^{\infty}(\R^d_x;C^N(\mathbb{S}_{\xi}^{d-1}))$.
\end{lemma}

\begin{proof}
1. We start with the proof of \eqref{stat. phase I}.
As in the proof of Theorem \ref{lemma Agmon--Hörmander complex} it suffices to consider the case $a=1$ and to replace $\E(\lambda)$ by \eqref{replacement of E(lambda)}.
Note that $\psi(0)=0$, $\nabla_{\xi'}\psi(0)=0$. The kernel $K_{1,\lambda}(x-y)$ of $\mathcal{E}_1(\lambda)\mathcal{E}_1(\overline{\lambda})^*$ is given by
\begin{align}\label{Appendix A def. K1lambda}
K_{1,\lambda}(z):=\int_{\R^{d-1}}\e^{\I \lambda |z|\phi(\omega,\xi')}\chi(\xi')^2\rd \xi',\quad \omega:=z/|z|\in \mathbb{S}^{d-1},
\end{align}
where $\phi(\omega,\xi'):=\omega\cdot(\psi(\xi'),\xi'))$. For $\lambda>0$, this is essentially the Fourier transform of the surface measure of $\lambda \mathbb{S}^{d-1}$, and the proof of \eqref{stat. phase I} is classical, see e.g.\ \cite[VIII.3.1.1]{MR1232192}. The proof for $\lambda\in\C$ is a simple modification, but we shall provide some details. The phase $\phi$ is stationary at the points $(\pm e_1,0)$, i.e.\ $\nabla_{\xi'}\phi(\pm e_1,0)=0$, where $e_1=(1,0,\ldots,0)$. Since 
\begin{align}\label{nondeg. Hessian}
\det\left[D^2_{\xi'\xi'}\phi(\omega,\xi')\right]=\omega_1^{d-1}\det \left[D^2_{\xi'\xi'}\psi(\xi')\right]\neq 0
\end{align}
at $(\omega,\xi')=(\pm 1,0)$, the implicit function theorem implies that, for $\omega$ in a small neighborhood $\mathcal{N}_{\pm}$ of $\pm e_1$, there is a unique solution $\xi'_{\pm}=\xi'_{\pm}(\omega)$ of the critical point equation $\nabla_{\xi'}\phi(\omega,\xi')=0$ (assuming, as we may, that the support of $\chi_1$ is sufficiently small). Moreover, if $\mathcal{N}_{\pm}$ are sufficiently small, then~\eqref{nondeg. Hessian} still holds at $(\omega,\xi'_{\pm}(\omega))$. 

2. We will only consider the two cases $\omega\in \mathcal{N}_{+}$ and $\omega\in \mathbb{S}^{d-1}\setminus (\mathcal{N}_{+}\cup \mathcal{N}_{-})$. The proof for $\omega\in \mathcal{N}_{-}$ is the same as for $\omega\in \mathcal{N}_{+}$. 
If $\omega\in \mathbb{S}^{d-1}\setminus (\mathcal{N}_{+}\cup \mathcal{N}_{-})$, then $\phi$ has no stationary point (if the support of $\chi$ is chosen sufficiently small) since $\nabla_{\xi'}\phi(\omega,0)=\omega'$ and $|\omega'|\geq c>0$. In this case, integration by parts yields 
\begin{align}\label{Appendix A IBP K1lambda}
|K_{1,\lambda}(z)|\lesssim_N \e^{|\im\lambda||z|}(1+|\lambda||z|)^{-N}
\end{align}
for any $N>0$, where we used that $|\phi(\omega,\xi')|\leq 1$ and thus $\e^{\I\lambda|z|\phi(\omega,\xi')}\leq \e^{|\im\lambda||z|}$. If $\omega\in \mathcal{N}_{+}$, 
we Taylor expand $\phi$ around the critical point,
\begin{align}\label{Appendix A Taylor}
\phi(\omega,\xi')=\phi(\omega,\xi'(\omega))+\frac{1}{2}\langle \eta',D^2_{\xi'\xi'}\phi(\omega,\xi'(\omega))\eta'\rangle+\mathcal{O}(|\eta'|^3),
\end{align}
where $\xi'(\omega)=\xi_+'(\omega)$ and $\eta'=\xi'-\xi'(\omega)$ (note that the prime does not stand for a derivative here). 
Then \eqref{nondeg. Hessian} implies that $|\nabla_{\xi'}\phi(\omega,\xi')|\gtrsim |\eta'|$ on $\supp(\chi_1)$. We split the integral in~\eqref{Appendix A def. K1lambda} into two parts, $K_{1,\lambda}(z)=I+II$, where
\begin{align}\label{Appendix A I}
I:=\int_{\R^{d-1}}\e^{\I \lambda |z|\phi(\omega,\xi')}\chi(\xi')^2\rho(|\lambda z|^{1/2}(\xi'-\xi'(\omega)))\rd \xi',
\end{align}
and $\rho$ is a bump function that localizes near the origin. A simple base times height estimate then shows that $|I|\lesssim \e^{|\im\lambda||z|}|\lambda z|^{-\frac{d-1}{2}}$.
For the second integral
\begin{align}\label{Appendix A II}
II=\int_{\R^{d-1}}\e^{\I \lambda |z|\phi(\omega,\xi')}\chi(\xi')^2(1-\rho(|\lambda z|^{1/2}\eta'))\rd \xi'\quad (\eta'=\xi'-\xi'(\omega)),
\end{align}
integration by parts yields, for any integer $N>d-1$ (similarly as in the proof of \cite[Lemma 4.15]{MR3052498}),
\begin{align*}
|II|\lesssim \e^{|\im\lambda||z|}|\lambda z|^{-N}\int_{|\eta'|>|\lambda z|^{-1/2}}(|\lambda z|^{N/2}|\eta'|^{-N}+|\eta'|^{-2N})\rd\eta'\lesssim \e^{|\im\lambda||z|}|\lambda z|^{-\frac{d-1}{2}}.
\end{align*} 
This proves $|K_{1,\lambda}(z)|\lesssim \e^{|\im\lambda||z|}(1+|\lambda z|)^{-\frac{d-1}{2}}$ for $|\lambda z|\geq 1$. The case $|\lambda z|< 1$ is of course trivial.

3. The proof of \eqref{stat. phase II} is similar.  
Instead of \eqref{Appendix A def. K1lambda}, we need to consider
\begin{align*}
\widetilde{K}_{1,\lambda}(x,y):=\int_{\R^{d-1}}\e^{\I (\lambda x-\overline{\lambda}y)\cdot(\psi(\xi'),\xi')}\chi(\xi')^2\rd \xi'.
\end{align*}
Note, however, that this is no longer a convolution kernel. We can write the phase as $|\lambda x-\overline{\lambda}y|\phi(\omega,\xi')$, with the same function $\phi(\omega,\xi')=\omega\cdot(\psi(\xi'),\xi'))$ as before, and with $\omega:=\frac{\lambda x-\overline{\lambda}y}{|\lambda x-\overline{\lambda}y|}\in \C^d$, $|\omega|=1$. If $|\omega'|\gtrsim 1$, then $|\nabla_{\xi'}\phi(\omega,\xi')|\gtrsim 1$ if the support of $\chi$ is sufficiently small, which we may assume. Integration by parts and the identity
\begin{align}\label{lambda xi-lambdabar eta}
    \lambda x-\overline{\lambda}y=\re\lambda(x-y)+\I\im\lambda(x+y)
\end{align}
then yield
\begin{align*}
|\widetilde{K}_{1,\lambda}(x,y)|\lesssim_N \e^{|\im\lambda||x+y|}(1+|\re\lambda||x-y|+|\im\lambda||x+y|)^{-N}.
\end{align*}
If $|\omega'|\ll 1$, then $|\omega_1|\gtrsim 1$. Since $\psi(\xi')=\sqrt{1-\xi'\cdot\xi'}$ has an analytic continuation to a neighborhood of the origin in $\C^{d-1}$ and $\psi(\xi')=1-\frac{1}{2}\xi'\cdot\xi'+\mathcal{O}(|\xi'|^3)$, there is a unique critical point of $\phi(\omega,\cdot)$ near the origin, given by $\xi(\omega)=\frac{\omega'}{\omega_1}(1+|\frac{\omega'}{\omega_1}|)\in\C^{d-1}$. Since  $D^2_{\xi'\xi'}\phi(\omega,\xi')=-\omega_1^{d-1}I+\mathcal{O}(|\xi'|)$ is nondegenerate there, \eqref{Appendix A Taylor} again implies $|\nabla_{\xi'}\phi(\omega,\xi')|\gtrsim |\xi'-\xi'(\omega)|$. Splitting $\widetilde{K}_{1,\lambda}(z)=I+II$, where now
\begin{align*}
I:=\int_{\R^{d-1}}\e^{\I |\lambda x-\overline{\lambda}y|\phi(\omega,\xi')}\chi(\xi')^2\rho(|\lambda x-\overline{\lambda}y|^{1/2}(\xi'-\re\xi'(\omega)))\rd \xi',
\end{align*} 
and repeating the estimates used for \eqref{Appendix A I}, \eqref{Appendix A II}, we obtain \eqref{stat. phase II}.
\end{proof}

\section{Singular value estimates}\label{sect. singular value estimates}
\subsection{Lorentz space potentials}
Assume that $V\in \e^{-2\gamma|\cdot|}L^{\frac{d+1}{2},\frac{1}{2}}(\R^d)$ for some $\gamma>0$.
Here, $\|\cdot\|_{p,q}$ denotes the quasinorm
\begin{align}\label{Lorentz norm}
\|V\|_{p,q}:=p^{\frac{1}{q}}\big(\int_0^{\infty}\alpha^q|\{x\in\R^d:\,|V(x)|>\alpha\}|^{\frac{q}{p}}\frac{\rd\alpha}{\alpha}\big)^{\frac{1}{q}}
\end{align}
in the Lorentz space $L^{p,q}(\R^d)$, $p,q\in(0,\infty)$. Note that $L^{p,p}(\R^d)=L^p(\R^d)$ and $L^{p,q}(\R^d)\subset L^p(\R^d)$ for $q<p$. 

\begin{proposition}\label{prop. sk(BS) Lp}
Suppose $d\geq 3$ is odd. Then 
\begin{align}\label{singular values Lp im lm>0}
s_k(BS(\lambda))\lesssim  k^{-\frac{1}{d+1}} |\lambda|^{-\frac{2}{d+1}}\|V\|_{(d+1)/2},\quad\im\lambda\geq 0,
\end{align}
and for any $\delta\in (0,1]$, $\gamma\geq \sqrt{1+\delta}|\im\lambda|$,
\begin{align}\label{singular values Lp im lm <0 (used for Schatten membership)}
s_k(BS(\lambda))\lesssim \delta^{-\frac{(d-1)^2}{d+1}} k^{-\frac{1}{d+1}} |\lambda|^{-\frac{2}{d+1}}\|\e^{2\gamma|\cdot|}V\|_{(d+1)/2,1/2},\quad \im\lambda <0.
\end{align}
Moreover, if $\delta,\epsilon\in (0,1]$,  $\gamma\geq \sqrt{1+\delta}|\im\lambda|+\epsilon|\lambda|$, then 
\begin{align}\label{singular values Lp exponential}
s_k(BS(\lambda)-BS(-\lambda))\lesssim |\lambda|^{-\frac{2}{d+1}}\delta^{-\frac{(d-1)^2}{d+1}}\exp\left(-c\epsilon k^{\frac{1}{d-1}}\right)\|\e^{2\gamma|\cdot|}V\|_{(d+1)/2,1/2}
\end{align}
for $\lambda\in\C$.
\end{proposition}

\begin{corollary}\label{cor. BS in weak Schatten}
If $\gamma>0$, $V\in \e^{-2\gamma|\cdot|}L^{(d+1)/2,1/2}(\R^d)$ and $\im\lambda>-\gamma$, then we have that $BS(\lambda)\in\mathfrak{S}^{d+1,\rm w}$. In particular, $BS(\lambda)$ is compact.
\end{corollary}

\begin{proof}
For $\im\lambda\geq 0$ the claim follows from 
\eqref{singular values Lp im lm>0}. If $-\gamma<\im\lambda<0$, then there exists $\delta>0$ such that $\gamma\geq \sqrt{1+\delta}|\im\lambda|$, and the claim follows from \eqref{singular values Lp im lm <0 (used for Schatten membership)}.
\end{proof}

\begin{proof}[Proof of Proposition \ref{prop. meromorphic Lp}]
  In view of the resolvent identity \eqref{resolvent identity} and since $BS(\lambda)$ is compact,
   this follows from a routine application of the meromorphic Fredholm theorem \cite[Thm. C.8]{MR3969938}.
\end{proof}

\begin{lemma}\label{lemma W1W2 bound E(lambda)E(lambdabar)star}
Let $d\geq 2$. Then for all $\lambda\in\C$, $\delta\in (0,1]$,
\begin{align}\label{W1W2 bound E(lambda)E(lambdabar)star}
s_k(W_1\rho_{\delta}\mathcal{E}_a(\lambda)\mathcal{E}_a(\overline{\lambda})^*\rho_{\delta}W_2)\lesssim c_{\delta} k^{-\frac{1}{d+1}} |\lambda|^{-\frac{d(d-1)}{d+1}}\|W_1\|_{d+1,1}\|W_2\|_{d+1,1},
\end{align}
uniformly over $a\in L^{\infty}(\R^d_x;C^N(\mathbb{S}_{\xi}^{d-1}))$ in the unit ball. Here, 
\begin{align}
    c_{\delta}:=\delta^{-\frac{(d-1)^2}{d+1}},\quad \rho_{\delta}(x,\lambda):=\e^{-\sqrt{1+\delta}|\im\lambda||x|}.
\end{align}
\end{lemma}

\begin{proof}
1. By scaling, we may assume that $|\lambda|=1$. 
Then \eqref{W1W2 bound E(lambda)E(lambdabar)star} would follow from the restricted weak type inequality
\begin{align}\label{FLS extension sj}
s_k(1_{\Omega_1}\rho_{\delta}\mathcal{E}_a(\lambda)\mathcal{E}_a(\overline{\lambda})^*\rho_{\delta}1_{\Omega_2})\lesssim c_{\delta}k^{-\frac{1}{d+1}}|\Omega_1|^{\frac{1}{d+1}}|\Omega_2|^{\frac{1}{d+1}},
\end{align}
where 
$\Omega_1,\Omega_2\subset\R^d$ are measurable sets. To see that \eqref{FLS extension sj} implies \eqref{W1W2 bound E(lambda)E(lambdabar)star}, one observes that, by the layer cake representation of $|W_1|,|W_2|$,
\begin{align}
    |W_j|=\int_0^{\infty}\mathbf{1}\{|W_j|>\alpha_j\}\rd \alpha_j,
\end{align}

the triangle inequality for the weak Schatten norm $\|\cdot\|_{\mathfrak{S}^{d+1,\rm w}}$ (see Section \ref{sect. Preliminaries} for the definition and the discussion that this is a norm) and the definition of the Lorentz norm $\|\cdot\|_{d+1,1}$, 
\begin{align*}
&\|W_1\rho_{\delta}\mathcal{E}_a(\lambda)\mathcal{E}_a(\overline{\lambda})^*\rho_{\delta}W_2\|_{\mathfrak{S}^{d+1,\rm w}}\\
&\leq\int_0^{\infty}\int_0^{\infty}\|\mathbf{1}\{|W_1|>\alpha_1\}\rho_{\delta}\mathcal{E}_a(\lambda)\mathcal{E}_a(\overline{\lambda})^*\rho_{\delta}\mathbf{1}\{|W_2|>\alpha_2\}\|_{\mathfrak{S}^{d+1,\rm w}}\rd\alpha_1\rd\alpha_2\\
&\lesssim c_{\delta} k^{-\frac{1}{d+1}}\int_0^{\infty}\int_0^{\infty} |\{|W_1|>\alpha_1\}|^{\frac{1}{d+1}}|\{|W_2|>\alpha_2\}|^{\frac{1}{d+1}}\rd\alpha_1\rd\alpha_2\\
&=c_{\delta}k^{-\frac{1}{d+1}}\|W_1\|_{d+1,1}\|W_2\|_{d+1,1}.
\end{align*}
Then \eqref{W1W2 bound E(lambda)E(lambdabar)star} follows from the equivalence of $\|\cdot\|_{\mathfrak{S}^{d+1,\rm w}}$ and $\|\cdot\|_{\mathfrak{S}^{d+1,\rm w}}'$

2.
Let $\{Q_{\alpha}\}_{\alpha\in \Z^d}$ be a partition of $\R^d$ into cubes of sidelength $R$ and denote by $\chi_{\alpha}$ the corresponding characteristic functions. Let 
\begin{align}
\Sigma_>&:=\sum_{|\alpha-\beta|>1}\chi_{\alpha}\rho_{\delta}\mathcal{E}_a(\lambda)\mathcal{E}_a(\overline{\lambda})^*\rho_{\delta}\chi_{\beta},\label{def. Sigma>}\\
\Sigma_{\leq}&:=\sum_{|\alpha-\beta|\leq 1}\chi_{\alpha}\rho_{\delta}\mathcal{E}_a(\lambda)\mathcal{E}_a(\overline{\lambda})^*\rho_{\delta}\chi_{\beta}.\label{def. Sigma_leq}
\end{align}
Then we have
\begin{align}\label{Ky Fan}
s_k(1_{\Omega_1}\rho_{\delta}\mathcal{E}_a(\lambda)\mathcal{E}_a(\overline{\lambda})^*\rho_{\delta}1_{\Omega_2})\leq s_k(1_{\Omega_1}\Sigma_>1_{\Omega_2})+\|1_{\Omega_1}\Sigma_{\leq}1_{\Omega_2}\|.
\end{align}
To estimate the second term we use Proposition \ref{lemma Agmon--Hörmander complex} (observing that the assumptions and the right hand side of \eqref{Agmon--Hörmander complex} are symmetric under complex conjugation), which gives
\begin{align}\label{Estimate Sigma2}
\|\chi_{\alpha}\rho_{\delta}\mathcal{E}_a(\lambda)\|,\|\mathcal{E}_a(\overline{\lambda})^*\rho_{\delta}\chi_{\beta}\|\lesssim \delta^{-\frac{d-1}{2}} R^{\frac{1}{2}}.
\end{align}
Combining \eqref{Estimate Sigma2} with the triangle inequality and Cauchy--Schwarz, we have that for any $f,g\in L^2(\R^d)$,
\begin{align*}
|\langle f, \Sigma_{\leq}g\rangle|&\leq \sum_{|\alpha-\beta|\leq 1}\|\mathcal{E}_a(\lambda)^*\rho_{\delta}\chi_{\alpha}f\|\|\mathcal{E}_a(\overline{\lambda})^*\rho_{\delta}\chi_{\beta}g\|\\
&\lesssim(\sum_{\alpha}\|\mathcal{E}_a(\lambda)^*\rho_{\delta}\chi_{\alpha}f\|^2)^{1/2}(\sum_{\alpha}\|\mathcal{E}_a(\overline{\lambda})^*\rho_{\delta}\chi_{\alpha}g\|^2)^{1/2}\\
&\lesssim R\delta^{-(d-1)}(\sum_{\alpha}\|\chi_{\alpha}f\|^2)^{1/2}(\sum_{\alpha}\|\chi_{\alpha}g\|^2)^{1/2}=R\delta^{-(d-1)}\|f\|\|g\|,
\end{align*}
which implies
\begin{align}\label{bound Sigma_2}
\|\Sigma_{\leq}\|\lesssim R\delta^{-(d-1)}.
\end{align}

3.
To estimate $s_k(1_{\Omega_1}\Sigma_>1_{\Omega_2})$ we use Markov's inequality
\begin{align}\label{Markov}
|\{s_k(1_{\Omega_1}\Sigma_>1_{\Omega_2})>\alpha\}|\leq \alpha^{-2}\sum_{n\in\N}s_n(1_{\Omega_1}\Sigma_>1_{\Omega_2})^2=\alpha^{-2}\|1_{\Omega_1}\Sigma_>1_{\Omega_2}\|_{\mathfrak{S}^2}^2.
\end{align}
The Hilbert--Schmidt norm is estimated using \eqref{stat. phase I},
which yields
\begin{align}\label{HS norm Sigma<}
|\Sigma_>(x-y)|\lesssim R^{-\frac{d-1}{2}}.
\end{align}
Here we used that $\e^{|\im\lambda|(|x-y|-\sqrt{1+\delta}(|x|+|y|))}\leq 1$ and $|x-y|>R$ on the support of $\Sigma_>$.
Thus we have
\begin{align*}
\|1_{\Omega_1}\Sigma_>1_{\Omega_2}\|_{\mathfrak{S}^2}^2
\lesssim R^{-(d-1)}
\int_{\Omega_2}\int_{\Omega_1}\big(\sum_{|\alpha-\beta|>1}\chi_{\alpha}(x)\chi_{\beta}(y)\big)^2\rd x\rd y.
\end{align*}
Since the cubes are almost disjoint (they intersect at most in a set of measure zero), we have
\begin{align}\label{Sigma_1 HS norm}
\|1_{\Omega_1}\Sigma_>1_{\Omega_2}\|_{\mathfrak{S}^2}^2\lesssim R^{-(d-1)}\sum_{\alpha,\beta}|\Omega_1\cap Q_{\alpha}||\Omega_2\cap Q_{\beta}|\leq R^{-(d-1)}|\Omega_1||\Omega_2|.
\end{align}

4.
Together, \eqref{Markov} and \eqref{Sigma_1 HS norm} imply
\begin{align}\label{bound Sigma_1}
s_k(1_{\Omega_1}\Sigma_>1_{\Omega_2})\lesssim k^{-\frac{1}{2}}R^{-\frac{d-1}{2}}|\Omega_1|^{\frac{1}{2}}|\Omega_2|^{\frac{1}{2}}.
\end{align}
Combining \eqref{Ky Fan}, \eqref{bound Sigma_2}, \eqref{bound Sigma_1}, we get
\begin{align*}
s_k(1_{\Omega_1}\rho_{\delta}\mathcal{E}_a(\lambda)\mathcal{E}_a(\overline{\lambda})^*\rho_{\delta}1_{\Omega_2})\lesssim \delta^{-(d-1)} R+k^{-\frac{1}{2}}R^{-\frac{d-1}{2}}|\Omega_1|^{\frac{1}{2}}|\Omega_2|^{\frac{1}{2}}.
\end{align*}
Optimizing over $R$, i.e. taking $R=\delta^{\frac{2(d-1)}{d+1}} k^{-\frac{1}{d+1}}|\Omega_1|^{\frac{1}{d+1}}|\Omega_2|^{\frac{1}{d+1}}$, yields  \eqref{FLS extension sj}. 
\end{proof}

We will also need the following variant of \eqref{W1W2 bound E(lambda)E(lambdabar)star}.

\begin{lemma}\label{lemma singular values 1 variant}
Suppose $d\geq 2$. Then for all $\lambda\in\C$, $\delta\in (0,1]$,
\begin{align}\label{eq. singular values 1 variant}
s_k(W\rho_{\delta}\mathcal{E}_a(\lambda))\lesssim c_{\delta}^{1/2} k^{-\frac{1}{2(d+1)}} |\lambda|^{-\frac{d(d-1)}{2(d+1)}}\|W\|_{d+1,1},
\end{align}
uniformly over $a$ in the unit ball of $L^{\infty}(\R^d_x;C^N(\mathbb{S}_{\xi}^{d-1}))$.
\end{lemma}

\begin{proof}
If $\lambda\in\R$, then \eqref{eq. singular values 1 variant} follows from \eqref{singular values Lp im lm>0} by a $TT^*$ argument. If $\lambda\notin\R$, we distinguish the cases $|\re\lambda|\gtrsim |\lambda|$ and $|\re\lambda|\ll |\lambda|$. By scaling, we may again assume $|\lambda|=1$, so that we consider the cases $|\re\lambda|\gtrsim 1$ and $|\im\lambda|\gtrsim 1$. It suffices to prove
\begin{align*}
s_k(1_{\Omega_1}\rho_{\delta}\mathcal{E}_a(\lambda)\mathcal{E}_a(\lambda)^*\rho_{\delta}1_{\Omega_2})\lesssim c_{\delta}|\Omega_1|^{\frac{1}{d+1}}|\Omega_2|^{\frac{1}{d+1}},
\end{align*}
This is the same as \eqref{FLS extension sj}, but without the complex conjugate. If $|\re\lambda|\gtrsim 1$, the proof is exactly the same as that of Lemma \ref{lemma W1W2 bound E(lambda)E(lambdabar)star}, except that we now use \eqref{stat. phase II} to get \eqref{HS norm Sigma<}. If If $|\im\lambda|\gtrsim 1$, we replace $|\alpha-\beta|$ by $|\alpha+\beta|$ in the definition of $\Sigma_>,\Sigma_{\leq}$. Then \eqref{stat. phase II} again yields \eqref{HS norm Sigma<}, and the estimate \eqref{bound Sigma_2} for $\Sigma_{\leq}$ is unchanged, due to the first bound in \eqref{Estimate Sigma2}.
\end{proof}

\begin{lemma}\label{lemma singular values WE exponential decay in sk}
Suppose $d\geq 2$. Then there exists $c>0$ such that for all $\lambda\in\C$ and $\delta,\epsilon\in (0,1]$,
\begin{align}\label{singular values complex extension}
s_k(W\rho_{\delta,\epsilon}\mathcal{E}(\lambda))\lesssim c_{\delta}^{1/2} \exp\left(-c\epsilon k^{\frac{1}{d-1}}\right)|\lambda|^{-\frac{d(d-1)}{2(d+1)}}\|W\|_{d+1,1},
\end{align}
where 
\begin{align}
    \rho_{\delta,\epsilon}(x,\lambda):=\e^{-(\sqrt{1+\delta}|\im\lambda|+\epsilon|\lambda|)|x|}.
\end{align}
\end{lemma}

\begin{proof}
For any $l\in\N$ we have
\begin{align}\label{trick sphere Laplacian}
s_k(W\rho_{\delta,\epsilon}\mathcal{E}(\lambda))\leq \| W\rho_{\delta,\epsilon}\mathcal{E}(\lambda)(I-\epsilon^2\Delta_S)^l\|s_k((I-\epsilon^2\Delta_S)^{-l})
\end{align}
where $\Delta_S$ is the Laplace--Beltrami operator on the sphere $S=\mathbb{S}^{d-1}$. By Weyl's asymptotics, 
\begin{align}\label{Weyl's asymptotics}
s_k((I-\epsilon^2\Delta_S)^{-l})\leq (C_w\epsilon k^{\frac{1}{d-1}})^{-2l}. 
\end{align} 
To estimate the operator norm in \eqref{trick sphere Laplacian} we use the \emph{structure formula}
\begin{align}\label{structure formula}
\mathcal{E}(\lambda)(I-\epsilon^2\Delta_S)^l=\mathcal{E}_{a_l}(\lambda),
\end{align}
where $a_l$ satisfies 
\begin{align}\label{al bound structure formula}
\sum_{|\beta|\leq N}|\partial_{\xi}^{\beta}a_l(x,\xi)|\leq C_{N,d}^{2l}(2l)!\exp(\epsilon|\lambda x|),\quad x\in\R^d,\quad \xi\in S
\end{align}
for any $N\in \N_0$. The proof of \eqref{structure formula}, \eqref{al bound structure formula} is postponed to Section \ref{Appendix Proof of the structure formula}. 
Lemma~\ref{lemma singular values 1 variant} (with $k=1$) then implies
\begin{align}\label{FLS extension al}
\| W\rho_{\delta,\epsilon}\mathcal{E}_{a_l}(\lambda)\|\lesssim c_{\delta}^{1/2} C_{N,d}^{2l}(2l)^{2l}|\lambda|^{-\frac{d(d-1)}{2(d+1)}}\|W\|_{d+1,1},
\end{align}
where we also used that $(2l)!\leq (2l)^{2l}$ and $s_1(K)=\|K\|$. Combining \eqref{trick sphere Laplacian}, \eqref{Weyl's asymptotics}, \eqref{structure formula}, \eqref{FLS extension al}, we obtain
\begin{align*}
s_k(W\rho_{\delta,\epsilon}\mathcal{E}(\lambda))\lesssim c_{\delta}^{1/2} (2l)^{2l}\left(C_wC_{N,d}^{-1}\epsilon k^{\frac{1}{d-1}}\right)^{-2l} |\lambda|^{-\frac{d(d-1)}{2(d+1)}}\|W\|_{d+1,1}.
\end{align*}
Since $\inf_{l\in\N}(2l)^{2l}M^{-2l}\lesssim \exp(-cM)$ for any $M>0$ and some $c>0$, the claim follows.
\end{proof}

\begin{proof}[Proof of Proposition \ref{prop. sk(BS) Lp}]
The estimate \eqref{singular values Lp im lm>0} for $\im\lambda\geq 0$ follows from \eqref{Frank--Laptev--Safronov}. 
For $\im\lambda<0$ we use Stone's formula \eqref{BS Stone formula}.
In view of \eqref{singular values Lp im lm>0} for $\im\lambda\geq 0$, \eqref{singular values Lp im lm <0 (used for Schatten membership)} follows from
\begin{align*}
s_k(|V|^{1/2}\rho_{\delta}\mathcal{E}(\lambda)\mathcal{E}(\overline{\lambda})^*\rho_{\delta}|V|^{1/2})\lesssim c_{\delta} k^{-\frac{1}{d+1}} |\lambda|^{-\frac{d(d-1)}{d+1}}\|V\|_{(d+1)/2,1/2},
\end{align*}
which is a consequence of the more general bound in Lemma~\ref{lemma W1W2 bound E(lambda)E(lambdabar)star}.
Similarly, \eqref{singular values Lp exponential} follows from
\begin{align*}
s_k(|V|^{1/2}\rho_{\delta,\epsilon}\mathcal{E}(\lambda)\mathcal{E}(\overline{\lambda})^*\rho_{\delta,\epsilon}|V|^{1/2})\lesssim c_{\delta} k^{-\frac{1}{d+1}} |\lambda|^{-\frac{d(d-1)}{d+1}}\exp\left(-c\epsilon k^{\frac{1}{d-1}}\right)\|V\|_{(d+1)/2,1/2},
\end{align*}
which is a consequence of Lemma \ref{lemma singular values WE exponential decay in sk}.
\end{proof}

\subsection{Compactly supported potentials}

\begin{proposition}\label{prop. sk(BS) comp.}
Suppose $d\geq 3$ is odd. Let $V\in L^{\infty}_{\rm comp}(\R^d)$ be supported in a ball $B_R=B(0,R)$. 
 Then
\begin{align}\label{cpt. singular values Ancona}
s_k(BS(\lambda))
\lesssim  (R|\lambda|+k^{\frac{1}{d}})^{-1}\|V\|_{\infty}R^2,\quad \im\lambda\geq 0,
\end{align}
and for $\lambda$ in the exterior of a fixed cone, 
\begin{align}\label{cpt. singular values exterior of cone}
s_k(BS(\lambda))
\lesssim (k^{\frac{2}{d}}R^{-2}+|\lambda|^2)^{-1}\|V\|_{\infty},\quad \im\lambda\gtrsim |\lambda|.
\end{align}
In particular, for any $\theta\in [0,1]$ and $\im\lambda\gtrsim |\lambda|$,
\begin{align}\label{cpt. singular values exterior of cone theta}
s_k(BS(\lambda))
\lesssim k^{\frac{-2\theta}{d}}R^{2\theta}|\lambda|^{-2(1-\theta)}\|V\|_{\infty}.
\end{align}
If $\nu\geq 0$, $\theta\in [0,1]$, then for $\im\lambda\geq 0$,
\begin{align}\label{cpt. singular values im lm >0}
s_k(BS(\lambda))\lesssim (k^{-\frac{\nu}{d-1}}  \langle \lambda R\rangle^{\nu}\ln \langle \lambda R\rangle R|\lambda|^{-1}+k^{\frac{-2\theta}{d}}R^{2\theta}|\lambda|^{-2(1-\theta)})\|V\|_{\infty},
\end{align}
and for any $\delta\in (0,1]$, $\im\lambda<0$,
\begin{align}\label{cpt. singular values im lm<0}
s_k(BS(\lambda)-BS(-\lambda))\lesssim k^{-\frac{\nu}{d-1}}\delta^{1-d}  \langle \lambda R\rangle^{\nu}\ln \langle \lambda R\rangle R|\lambda|^{-1}\|\rho_{\delta}^{-2}V\|_{\infty}.
\end{align}
Moreover, for any $\delta,\epsilon\in (0,1]$, $\im\lambda<0$, 
\begin{align}\label{cpt. singular values exponential}
s_k(BS(\lambda)-BS(-\lambda))\lesssim \delta^{1-d}R|\lambda|^{-1}\|\rho_{\delta,\epsilon}^{-2}V\|_{\infty}\exp\left(-c\epsilon k^{\frac{1}{d-1}}\right).
\end{align}
All implicit constant depend only on $d,\theta,\nu$.
\end{proposition}

\begin{proof}
1. The proof strategy is similar to that of Proposition \ref{prop. sk(BS) Lp}. We first note that
\begin{align*}
(k^{\frac{2}{d}}R^{-2}+\lambda^2)^{-1}\leq 
k^{\frac{-2\theta}{d}}R^{2\theta}|\lambda|^{-2(1-\theta)}
\end{align*}
for any $\theta\in[0,1]$. Hence \eqref{cpt. singular values exterior of cone theta} follows from \eqref{cpt. singular values exterior of cone}. For the proof of \eqref{cpt. singular values im lm >0} it will be convenient to split $R_0(\lambda)=R_0(\lambda)^{\rm high}+R_0(\lambda)^{\rm low}$, where $R_0(\lambda)^{\rm low}$ has Fourier multiplier $(\xi^2-\lambda^2)^{-1}\chi(\xi/|\lambda|)$ and $\chi$ is a smooth, radial function on $\R^d$ that equals $1$ on $B(0,2)$ and is supported on $B(0,4)$. In view of Stone's formula \eqref{BS Stone formula} and since $|V|\leq \|V\|_{\infty}\mathbf{1}_{B_R}$, it will then suffice to prove 
\begin{align}\label{cpt. singular values R0 lm>0}
s_k(\mathbf{1}_{B_R}R_0(\lambda)\mathbf{1}_{B_R})
\lesssim (R|\lambda|+k^{\frac{1}{d}})^{-1}R^2,\quad \im\lambda\geq 0,
\end{align}
\begin{align}\label{cpt. singular values R0 exterior cone}
s_k(\mathbf{1}_{B_R}R_0(\lambda)\mathbf{1}_{B_R})
\lesssim (k^{\frac{2}{d}}R^{-2}+\lambda^2)^{-1},\quad \im\lambda\gtrsim|\lambda|,
\end{align}
\begin{align}\label{cpt. singular values R0high}
s_k(\mathbf{1}_{B_R}R_0(\lambda)^{\rm high}\mathbf{1}_{B_R})
\lesssim (k^{\frac{2}{d}}R^{-2}+\lambda^2)^{-1},\quad \im\lambda\geq 0,
\end{align}
\begin{align}\label{cpt. singular values Elambda R0low}
s_k(\mathbf{1}_{B_R}R_0(\lambda)^{\rm low}\mathbf{1}_{B_R})
\lesssim k^{-\frac{\nu}{d-1}} R|\lambda|^{-1}\langle \lambda R\rangle^{\nu}\ln \langle \lambda R\rangle,\quad \im\lambda\geq 0,
\end{align}
\begin{align}\label{cpt. singular values Elambda}
s_k(\mathbf{1}_{B_R}\rho_{\delta}\mathcal{E}(\lambda))\lesssim  k^{-\frac{\nu}{2(d-1)}} R^{\frac{1}{2}}|\delta\lambda|^{-\frac{d-1}{2}}\langle \lambda R\rangle^{\frac{\nu}{2}},\quad\lambda\in\C,
\end{align}
\begin{align}\label{cpt. singular values Elambda exponential}
s_k(\mathbf{1}_{B_R}\rho_{\delta,\epsilon}\mathcal{E}(\lambda))\lesssim  
R^{\frac{1}{2}}|\delta\lambda|^{-\frac{d-1}{2}}\exp\left(-c\epsilon k^{\frac{1}{d-1}}\right),\quad\lambda\in\C.
\end{align}

2. We first prove \eqref{cpt. singular values R0 lm>0}
We start with the following resolvent estimate, see e.g. \cite[(3.3)]{MR4025959}:
\begin{align}\label{D'Ancona}
|\lambda|\|R_0(\lambda)f\|_{\dot{Y}}+\|\nabla R_0(\lambda)f\|_{\dot{Y}}\lesssim \|f\|_{\dot{Y}^*},\quad \im\lambda\geq 0,
\end{align} 
where the space $\dot{Y}$ is defined through the norm
\begin{align*}
\|f\|_{\dot{Y}}^2:=\sup_{R>0}\frac{1}{R}\int_{|x|\leq R}|f(x)|^2\rd x,
\end{align*}
and the (pre-) dual $\dot{Y}^*$ can be equipped with the equivalent norm
\begin{align*}
\|f\|_{\dot{Y}^*}\asymp \sum_{j\in\Z} 2^{j/2}\|f\|_{L^2(2^j\leq |x|<2^{j+1})}.
\end{align*}
Note that $\dot{Y}^*$ can be viewed as a homogeneous version of the ``$B$-space" of Agmon and Hörmander \cite{MR0466902}. We have
\begin{align*}
s_k(\mathbf{1}_{B_R}R_0(\lambda)\mathbf{1}_{B_R})\leq s_k((|\lambda|^2-\Delta_{\mathbb{T_R}})^{-\frac{1}{2}})\|(|\lambda|^2-\Delta_{\mathbb{T_R}})^{\frac{1}{2}}\mathbf{1}_{B_R}R_0(\lambda)\mathbf{1}_{B_R}\|,
\end{align*}
where $\mathbb{T}_R:=\R^d/(2R\Z)^d$ is the $d$-dimensional torus of sidelength $2R$. By Weyl's asymptotics,
\begin{align*}
s_k((|\lambda|^2-\Delta_{\mathbb{T_R}})^{-\frac{1}{2}})\lesssim (|\lambda|^2+R^{-2}k^{\frac{2}{d}})^{-\frac{1}{2}}.
\end{align*}
Using \eqref{D'Ancona}, we have
\begin{align*}
&\|(|\lambda|^2-\Delta_{\mathbb{T_R}})^{\frac{1}{2}}\mathbf{1}_{B_{R}}R_0(\lambda)\mathbf{1}_{B_{R}}\|\lesssim |\lambda| \|\mathbf{1}_{B_{R}}R_0(\lambda)\mathbf{1}_{B_{R}}\|+\|\mathbf{1}_{B_{2R}}\nabla R_0(\lambda)\mathbf{1}_{B_{2R}}\|\\
&+R^{-1}\|\mathbf{1}_{B_{2R}}\nabla R_0(\lambda)\mathbf{1}_{B_{2R}}\|\lesssim R.
\end{align*}
In the first inequality we used that $\|(|\lambda|^2-\Delta_{\mathbb{T_R}})^{\frac{1}{2}}\cdot\|$ is equivalent to $|\lambda|\|\cdot\|+\|\nabla\cdot\|$ and inserted a smooth cutoff function $\chi_R=\chi(\cdot/R)$, with $\mathbf{1}_{B_R}\leq \chi_R\leq \mathbf{1}_{B_{2R}}$. Combining the last two displays proves \eqref{cpt. singular values R0 lm>0} and hence \eqref{cpt. singular values Ancona}. 

3. Next, we prove \eqref{cpt. singular values R0 exterior cone}, \eqref{cpt. singular values R0high}, \eqref{cpt. singular values Elambda R0low}.
Since the Fourier multiplier of $R_0(\lambda)^{\rm high}$ is bounded by $(\xi^2+\lambda^2)^{-1}$, Weyl asymptotics for $-\Delta$ on $B_R$ yield \eqref{cpt. singular values R0high}.
The same estimate holds for $R_0(\lambda)$ instead of $R_0(\lambda)^{\rm high}$ if $\im\lambda\gtrsim|\lambda|$, which proves \eqref{cpt. singular values R0 exterior cone}.
We will momentarily show that the estimate \eqref{cpt. singular values Elambda R0low} for $R_0(\lambda)^{\rm low}$ would follow from 
\begin{align}\label{cpt. singular values Elambda 4}
s_k(\mathbf{1}_{B_R}\mathcal{E}(\lambda))\lesssim  k^{-\frac{\nu}{2(d-1)}} R^{\frac{1}{2}}|\lambda|^{-\frac{d-1}{2}}\langle \lambda R\rangle^{\frac{\nu}{2}},
\quad \im\lambda\geq 0.
\end{align}
We defer the proof of \eqref{cpt. singular values Elambda 4} to Lemma \ref{lemma comp. sk E(lambda)} below.
By scaling, we may assume that $\re \lambda=1$. Moreover, since we have already proved \eqref{cpt. singular values R0 exterior cone}, we may assume $0\leq \im\lambda<1$. Using polar coordinates, we have
\begin{align*}
R_0(1\pm\I\, \im\lambda)^{\rm low}f=(2\pi)^{-d}\int_0^{4}\frac{\chi(r)\E(r)\E(r)^*f}{r^2-(1\pm\I\, \im\lambda)^2}\rd r.
\end{align*} 
It would then follow from \eqref{cpt. singular values Elambda 4} that
\begin{align*}
\|\mathbf{1}_{B_R}R_0(1\pm\I\, \im\lambda)^{\rm low}\mathbf{1}_{B_R}\|_{\mathfrak{S}^{\frac{d-1}{\nu},\rm w}}\lesssim R(1+R)^{\nu} \int_0^{4}|r-1|^{-1}\rd r.
\end{align*}
Unfortunately, the integral is divergent. The redeeming feature is that the spatial localization to $B_R$ smooths out the integrand on the $1/R$ scale, i.e.\ $|r-1|$ may be replaced by $|r-1|+1/R$. To make this rigorous, one observes that, by the convolution theorem, 
\begin{align*}
\mathbf{1}_{B_R}m(D)\mathbf{1}_{B_R}=\mathbf{1}_{B_R}m_{R}(D)\mathbf{1}_{B_R},
\end{align*}
whenever $m(D)$ is a Fourier multiplier, $m_R:=\varphi_R\ast m$, $\varphi_R(\xi):=R^d\varphi(R\xi)$ and $\varphi$ is a Schwartz function such that $\widehat{\varphi}=1$ on $B(0,2)$. 
The previous argument thus yields (see \cite[Sect. 6.2]{CM22} for more details)
\begin{align}\label{ln appears}
\|\mathbf{1}_{B_R}R_0(1\pm\I 0)^{\rm low}\mathbf{1}_{B_R}\|_{\mathfrak{S}^{\frac{d-1}{\nu},\rm w}}\lesssim R\langle R\rangle^{\nu} \int_0^4 (|r-1|+1/R)^{-1}\rd r
\lesssim R\langle R\rangle^{\nu}\ln \langle R\rangle,
\end{align}
which completes the proof of \eqref{cpt. singular values Elambda R0low}.

4. We defer the proof of \eqref{cpt. singular values Elambda} to Lemma \ref{lemma comp. sk E(lambda)} below and  continue with the proof of \eqref{cpt. singular values Elambda exponential}. As in the proof of Lemma \ref{lemma singular values WE exponential decay in sk} we estimate
\begin{align*}
s_k(\mathbf{1}_{B_R}\rho_{\delta,\epsilon}\mathcal{E}(\lambda))\leq \|\mathbf{1}_{B_R}\rho_{\delta,\epsilon}\mathcal{E}(\lambda)(I-\epsilon^2\Delta_S)^l\|s_k((I-\epsilon^2\Delta_S)^{-l}).
\end{align*}
The same argument as there then shows that \eqref{cpt. singular values Elambda exponential} is a consequence of the norm bound \eqref{Agmon--Hörmander complex} and Weyl's asymptotics \eqref{Weyl's asymptotics}.
\end{proof}

\begin{lemma}\label{lemma comp. sk E(lambda)}
Suppose $d\geq 2$, $R>0$, $\nu\geq 0$. Then \eqref{cpt. singular values Elambda 4} holds. Moreover, \eqref{cpt. singular values Elambda} holds for any $\delta\in(0,1]$.
\end{lemma}

\begin{proof}
The proof is again similar to that of Lemma \ref{lemma singular values WE exponential decay in sk}. The argument is simpler since $\nu$ is fixed; however, $\nu$ is not necessarily an integer. We can still write
\begin{align}
s_k(\mathbf{1}_{B_R}\mathcal{E}(\lambda))&\leq \|\mathbf{1}_{B_R}\mathcal{E}(\lambda)(I-\Delta_S)^{\frac{\nu}{4}}\|s_k((I-\Delta_S)^{-\frac{\nu}{4}}),\label{nu/4 a}\\
s_k(\mathbf{1}_{B_R}\rho_{\delta}\mathcal{E}(\lambda))&\leq \|\mathbf{1}_{B_R}\rho_{\delta}\mathcal{E}(\lambda)(I-\Delta_S)^{\frac{\nu}{4}}\|s_k((I-\Delta_S)^{-\frac{\nu}{4}}),\label{nu/4 b}
\end{align}
where $(I-\Delta_S)^{\kappa}$ is defined in terms of the eigenfunction expansion of $-\Delta_S$. If $\nu/4$ is an integer, then
integration by parts together with \eqref{AH for lambda>0} yields
\begin{align*}
\|\mathbf{1}_{B_R}\mathcal{E}(\lambda)(I-\Delta_S)^{\frac{\nu}{4}}\|\lesssim 
R^{\frac{1}{2}}|\lambda|^{-\frac{d-1}{2}}\langle \lambda R\rangle^{\frac{\nu}{2}},\quad\lambda>0,
\end{align*}
which in combination with \eqref{nu/4 a} and Weyl's asymptotics \eqref{Weyl's asymptotics} yields \eqref{cpt. singular values Elambda 4} in this case. For non integer $\nu/4$, the result follows by complex interpolation applied to the analytic family of operators $\kappa\mapsto \langle \lambda R\rangle^{-2\kappa}(I-\Delta_S)^{\kappa}$. The proof of \eqref{cpt. singular values Elambda} is analogous, but uses \eqref{Agmon--Hörmander complex} instead of \eqref{AH for lambda>0}.
\end{proof}

\subsection{Pointwise decaying potentials}
We will use the following abbreviations:
\begin{align}
v_{\rho,R}:=\sup_{x\in\R^d}(1+|x|/R)^{\rho}|V(x)|,\quad
v_{\rho,R,\gamma}:=\sup_{x\in\R^d}\e^{2\gamma|x|}(1+|x|/R)^{\rho}|V(x)|.
\end{align}

\begin{proposition}\label{prop. sk(BS) ptw.}
Suppose $d\geq 3$ is odd, $\theta\in [0,1]$, $\rho,R,\gamma>0$.
If $\rho>2\theta$, then for $\lambda$ in the exterior of a fixed cone $\im\lambda\gtrsim |\lambda|$,
\begin{align}\label{ptw. singular values exterior of cone theta}
s_k(BS(\lambda))
\lesssim  k^{-\frac{2\theta}{d}}R^{2\theta}|\lambda|^{-2(1-\theta)}v_{\rho,R}.
\end{align}
If $\nu\geq 0$, $\kappa>0$ and $\rho>\max(1+\nu+\kappa,2\theta)$, then for $\im\lambda\geq 0$,
\begin{align}\label{ptw. singular values im lm >0}
s_k(BS(\lambda))\lesssim (k^{-\frac{\nu}{d-1}}  \langle \lambda R\rangle^{\nu+\kappa}R|\lambda|^{-1}+k^{\frac{-2\theta}{d}}R^{2\theta}|\lambda|^{-2(1-\theta)})v_{\rho,R}.
\end{align}
Moreover, for $\nu,\kappa,\rho$ as above, $\delta\in (0,1]$, $\gamma\geq \sqrt{1+\delta}|\im\lambda|$, $\im\lambda<0$, we have
\begin{align}\label{ptw. singular values im lm <0 (needed for Schatten membership)}
s_k(BS(\lambda))
\lesssim \delta^{1-d}k^{-\frac{\nu}{d-1}}  \langle \lambda R\rangle^{\nu+\kappa}R|\lambda|^{-1}v_{\rho,R,\gamma}+k^{\frac{-2\theta}{d}}R^{2\theta}|\lambda|^{-2(1-\theta)}v_{\rho,R},
\end{align}
and for $\delta,\epsilon\in (0,1]$, $\gamma\geq \sqrt{1+\delta}|\im\lambda|+\epsilon|\lambda|$, $\im\lambda<0$, we have
\begin{align}\label{ptw. singular values exponential}
s_k(BS(\lambda)-BS(-\lambda))\lesssim \delta^{1-d}R|\lambda|^{-1}v_{\rho,R,\gamma}\exp\left(-c\epsilon k^{\frac{1}{d-1}}\right).
\end{align}
All implicit constant depend only on $d,\theta,\rho,\nu,\kappa$.
\end{proposition}

\begin{proof}
The results follow from Proposition \ref{prop. sk(BS) comp.} by dyadic summation. More precisely, we have $V=\sum_j V_j$ where 
$V_j=V\mathbf{1}_{A_j}$ is supported in the dyadic shell $A_j=\{2^{j}R\leq |x|\leq  2^{j+1}R\}$, $j\in\N_0$.
Then \eqref{cpt. singular values exterior of cone theta} implies
\begin{align*}
    \|BS(\lambda)\|_{\mathfrak{S}^{\frac{d}{2\theta},\rm w}}&=\|R_0(\lambda)^{\frac{1}{2}}V|R_0(\lambda)|^{\frac{1}{2}}|\|_{\mathfrak{S}^{\frac{d}{2\theta},\rm w}}\leq \sum_j \|R_0(\lambda)^{\frac{1}{2}}V_j|R_0(\lambda)|^{\frac{1}{2}}|\|_{\mathfrak{S}^{\frac{d}{2\theta},\rm w}}\\
&=\sum_j \||V_j|^{\frac{1}{2}}|R_0(\lambda)V_j^{\frac{1}{2}}\|_{\mathfrak{S}^{\frac{d}{2\theta},\rm w}}
    \lesssim |\lambda|^{2(1-\theta)}R^{2\theta}v_{\rho,R} \sum_j 2^{(2\theta-\rho) j}.
\end{align*}
Since $\rho>2\theta$, this yields \eqref{ptw. singular values exterior of cone theta}. Similarly, dyadic summation of \eqref{cpt. singular values im lm >0}, \eqref{cpt. singular values im lm<0}, using $\ln\langle\lambda R\rangle\lesssim_{\kappa} \langle \lambda R\rangle^{\kappa}$, yields \eqref{ptw. singular values im lm >0}, \eqref{ptw. singular values im lm <0 (needed for Schatten membership)}. To prove \eqref{cpt. singular values exponential}, we observe that
\begin{align*}
\|\e^{2(\sqrt{1+\delta}|\im\lambda|+\epsilon|\lambda|)|\cdot|}V\|_{L^{\infty}(2^{j-1}R\leq |x|\leq  2^jR)}\leq v_{\rho,R,\gamma}2^{-\rho j}.
\end{align*}
Thus, dyadic summation of \eqref{cpt. singular values exponential} yields \eqref{ptw. singular values exponential}.
\end{proof}

\begin{corollary}\label{cor. BS in weak Schatten ptw.}
Assume that $\theta\in[1/2,1]$, $\rho>2\theta$ and $v_{\rho,R,\gamma}<\infty$. Then we have $BS(\lambda)\in\mathfrak{S}^{\alpha}$ for $\im\lambda>-\gamma$ and $\alpha>\max(\frac{d-1}{\rho-1},\frac{d}{2\theta})$. In particular, $BS(\lambda)$ is compact.
\end{corollary}

\begin{proof}
Assume $\im\lambda\geq 0$ first. Let $\kappa>0$ be such that $\nu:=\rho-1-2\kappa>0$ and $\alpha>\max(\frac{d-1}{\nu},\frac{d}{2\theta})$. Since $\rho>\max(1+\nu+\kappa,2\theta)$, \eqref{ptw. singular values im lm >0} implies that $BS(\lambda)\in\mathfrak{S}^{\alpha}$. If $0>\im\lambda>-\gamma$, then there exists $\delta>0$ such that $\gamma\geq\sqrt{1+\delta}|\im\lambda|$, and the claim follows from \eqref{ptw. singular values im lm <0 (needed for Schatten membership)} in the same way.
\end{proof}

The proof of Proposition \ref{prop. meromorphic ptw.} then follows from the same routine argument as that of Proposition \ref{prop. meromorphic Lp}.

\section{Fredholm determinant}\label{sect. Fredholm det.}
We place $D(0,r)$ into a larger disk $D(\lambda_0,|\lambda_0|+r)$, where $\lambda_0\in\I\R_+$ will be suitably chosen, and consider the function $F_{\alpha}$ defined in \eqref{def. F_alpha}. We recall the definitions of $H_{\alpha}(\lambda)$ and $F_{\alpha}(\lambda)$:
\begin{align*}
H_{\alpha}(\lambda):=\det(I-(-BS(\lambda))^{\alpha}),\quad F_{\alpha}(k):=\frac{H_{\alpha}(\lambda_0+k)}{H_{\alpha}(\lambda_0)},\quad \lambda,k\in\C.
\end{align*}
We also recall the definition of the `averaged counting functions'
\begin{align*}
N_V(r):=\int_0^r\frac{n_V(t)}{t}\rd t,\quad N_{H_{\alpha}}(r):=\int_0^r\frac{n_{H_{\alpha}}(t)}{t}\rd t,\quad \quad N_{F_{\alpha}}(r):=\int_0^r\frac{n_{F_{\alpha}}(t)}{t}\rd t,
\end{align*}
where $n_V(t)$ denotes the number of resonances of $-\Delta+V$ and $n_{H_{\alpha}}$, $n_{F_{\alpha}}$ denote the number of zeros (counting multiplicity) of the analytic functions $H_{\alpha}$, $F_{\alpha}$, respectively, in the disk $D(0,r)$. Finally, we recall the estimates $N_V(r)\leq N_{H_{\alpha}}(r)$
and $n_V(r)\leq (\ln s)^{-1}N_V(sr)$, for any fixed $s>1$.
\begin{figure}[b]
\begin{center}
\begin{tikzpicture}
\draw[gray, thick] (-3,0) -- (3,0);
\draw[gray, thick] (0,-1) -- (0,3);
\draw[black, very thick](0,0) circle (0.5);
\draw[black, very thick](0,1) circle (1.5);
\filldraw[black] (0,1) circle (2pt) node[anchor=west]{$|\lambda_0|$};
\filldraw[black] (1.118,0) circle (2pt)node[anchor=west]{$P$};
\draw[black, thick, dotted] (0,1) -- (1.118,0);
\draw[black, thick, dotted] (-3,-0.5) -- (3,-0.5)node[anchor=west]{$\im\lambda=-r$};
\end{tikzpicture}
\end{center}
\caption{The disk $D(0,r)$ inside the larger disk $D(\lambda_0,|\lambda_0|+r)$.}
\end{figure}
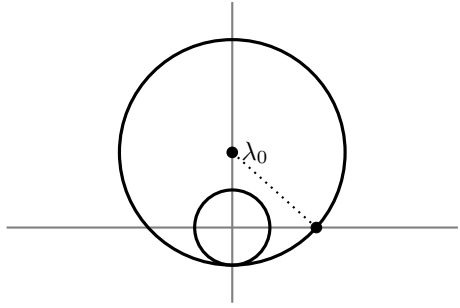
Hence, we have 
\begin{align}
    n_V(r)\ln s\leq N_V(sr)\leq N_{H_{\alpha}}(sr)\leq N_{F_{\alpha}}(|\lambda_0|+sr),
 \end{align}   
    where the last inequality follows from the triangle inequality. Indeed, if $\lambda\in D(0,sr)$ is a zero of $H_{\alpha}$, then $k=\lambda-\lambda_0$ is a zero of $F_{\alpha}$ and $|k|< |\lambda_0|+sr$.
By Jensen's formula, since $F_{\alpha}(0)=1$, we have
\begin{align*}
N_{F_{\alpha}}(|\lambda_0|+sr)&=\frac{1}{2\pi}\int_0^{2\pi}\ln|F_{\alpha}((|\lambda_0|+sr)\e^{\I\theta})|\rd \theta
\leq \max_{|k|=|\lambda_0|+sr}\ln|F_{\alpha}(k)|.
\end{align*}
By definition of $F_{\alpha}$,
\begin{align*}
    \max_{|k|=|\lambda_0|+sr}\ln|F_{\alpha}(k)|&=\max_{|k|=|\lambda_0|+sr}\ln|H_{\alpha}(\lambda_0+k)|-\ln|H_{\alpha}(\lambda_0)|\\
 &=\max_{\lambda\in \partial D(\lambda_0,|\lambda_0|+sr)}\ln|H_{\alpha}(\lambda)|-\ln|H_{\alpha}(\lambda_0)|
\end{align*}
To estimate this we set
\begin{align}
    S_+&:=S_+(sr,\lambda_0):=\{\lambda\in\C:\im\lambda\geq 0,\, |\lambda|\geq P\},\\
    S_-&:=S_-(sr,\lambda_0):=\{\lambda\in\C:-sr\leq \im\lambda< 0,\, sr\leq |\lambda|\leq P\},
\end{align}
where $P:=P(sr,\lambda_0):=\sqrt{2|\lambda_0|sr+(sr)}^2$.
Elementary geometry (see Figure 1) reveals that 
\begin{align}
    \partial D(\lambda_0,|\lambda_0|+sr)\subset S_+(sr,\lambda_0)\cup S_-(sr,\lambda_0).
\end{align}
In summary, we have 
\begin{align}\label{in summary N_V leq...}
    N_V(sr)\leq \max_{\lambda\in S_+\cup S_-}\ln|H_{\alpha}(\lambda)|-\ln|H_{\alpha}(\lambda_0)|,
\end{align}
where we recall that $S_{\pm}$ depend on $sr$ and $\lambda_0$, and $\lambda_0\in\I \R_+$ is to be determined.
By the Weyl inequality \cite[B.5.8]{MR3969938}, for any trace class operator $A$,
\begin{align}\label{Weyl inequality}
\ln|\det(I-A)|\leq \sum_{k\in\N}\ln(1+s_k(A))\leq \sum_{k\in\N} s_k(A),
\end{align}
where we used $\ln(1+x)\leq x$ for $x\geq 0$ in the second inequality. We use \eqref{Weyl inequality} with $A=(-BS(\lambda))^{\alpha}$.
By sparsifying the sum
\begin{align}\label{sparsifying}
    \sum_{k\in\N}=\sum_{k\in\alpha\N}+\sum_{k\in\alpha\N-1}+\ldots +\sum_{k\in\alpha(\N-1)+1}
\end{align}
and using the monotonicity of the singular values, it is sufficient to estimate the last sum in the above display, which can be written as
\begin{align}\label{last term in sparsifying}
    \sum_{j=0}^{\infty}\ln(1+s_{\alpha j+1}((-BS(\lambda))^{\alpha})).
\end{align}
Repeated application of \eqref{s_j(AB)} shows that 
\begin{align}\label{s_alpha j}
    s_{\alpha j+1}((-BS(\lambda))^{\alpha})\leq [s_{j+1}(BS(\lambda))]^{\alpha}.
\end{align}
We thus obtain the following estimate on the first term of \eqref{in summary N_V leq...}.
\begin{lemma}\label{lemma Weyl type first term}
Assume that $\alpha>d/2$ is an integer. Then for $\lambda\in S_+\cup S_-$ we have
\begin{align}\label{Weyl type first term}
\ln|H_{\alpha}(\lambda)|\leq \alpha \sum_{k\in\N}\ln(1+[s_k(BS(\lambda))]^{\alpha})\leq \alpha \|BS(\lambda)\|_{\mathfrak{S}^{\alpha}}^{\alpha}.
\end{align}
\end{lemma}
We also need an estimate for the second term in \eqref{in summary N_V leq...}. To this end we use \eqref{Weyl inequality} with $A=(I-K)^{-1}K$ and $K=(-BS(\lambda_0))^{\alpha}$. In view of the identity $(I-K)^{-1}=I+A$, this gives
\begin{align*}
|\det((I-K)^{-1})|\leq \prod_{k\in\N} (1+\|(I-K)^{-1}\|s_k(K)).
\end{align*}
Assuming that $\|K\|\leq 1/2$, we could estimate $\|(I-K)^{-1}\|\leq 2$. Then, using $\det((I-K)^{-1})=(\det(I-K))^{-1}$, we would have
\begin{align*}
-\ln|H_{\alpha}(\lambda_0)|\leq \sum_{k\in\N} \ln(1+2s_k(K)).
\end{align*}
The argument leading to \eqref{Weyl type first term} 
then yields the following estimate.
\begin{lemma}\label{lemma Weyl type second term}
Assume that $\alpha>d/2$ is an integer and that $\|BS(\lambda_0)^{\alpha}\|\leq 1/2$. Then
\begin{align}\label{Weyl type second term}
-\ln|H_{\alpha}(\lambda_0)|\leq \alpha\sum_{k\in\N}\ln(1+2[s_k(BS(\lambda_0))]^{\alpha})\leq 2\alpha \|BS(\lambda_0)\|_{\mathfrak{S}^{\alpha}}^{\alpha}.
\end{align} 
\end{lemma}
Since $\alpha>1$, the assumption of Lemma \ref{lemma Weyl type second term} would be satisfied if $\|BS(\lambda_0)\|\leq 1/2$. In view of the estimates \eqref{Frank--Laptev--Safronov}, \eqref{cpt. singular values exterior of cone} and \eqref{ptw. singular values exterior of cone theta} (with $k=1$ and $\theta=0$) this will be achieved if we set
\begin{align}\label{def lambda0}
2\lambda_0:=\begin{cases}
\I A\|V\|_{(d+1)/2}^{(d+1)/2} \quad &\mbox{(Lorentz space)},\\
\I A\|V\|_{\infty}^{\frac{1}{2}} \quad &\mbox{(compactly supported)},\\
\I A v_{\rho,R}^{\frac{1}{2}}  \quad &\mbox{(pointwise decaying)}.
\end{cases}
\end{align}
where $A$ is a sufficiently large dimensionless constant.
Here $\rho>1$ and $R>0$ are fixed. 
In conclusion, we state the bounds that will be needed in the proof of Theorem 7.1. By \eqref{in summary N_V leq...}, \eqref{Weyl type first term}, \eqref{Weyl type second term} we have 
\begin{align}\label{conclusion sect. 6-1}
    n_V(r)\ln s\leq N_V(sr)\lesssim \alpha(\tilde{n}_+(sr)+\tilde{n}_-(sr)+\tilde{n}_0)
\end{align}
where (fixing an integer $\alpha>d/2$ and $\lambda_0$ as in \eqref{def lambda0})
\begin{align}\label{conclusion sect. 6-2}
        \tilde{n}_{\pm}(r):=\max_{\lambda\in S_{\pm}(r,\lambda_0)} \sum_{k\in\N}\ln(1+[s_k(BS(\lambda))]^{\alpha}),\quad
         \tilde{n}_0:=\sum_{k\in\N}\ln(1+[s_k(BS(\lambda_0))]^{\alpha}).
\end{align}

\section{Proof of Theorem \ref{thm. main intro}}\label{sect. proof main thm. intro}
We will prove the following stronger versions of Theorem \ref{thm. main intro}.
\begin{theorem}\label{thm. main detailed}
    Suppose $d\geq 3$ is odd, $r,\gamma>0$ and
$\delta,\epsilon\in (0,1]$. In the compactly supported case, assume that $\supp V\subset B(0,R)$. In the pointwise decaying and compactly supported case, we also fix $\theta\in [1/2,1]$, $\nu\geq 0$, $\kappa>0$, and $\rho>\max(1+\nu+\kappa,2\theta)$ (the latter is only needed in the pointwise decaying case). 
\begin{enumerate}
    \item[(a)] Then every resonance $\lambda\in \C$ satisfies
\begin{align}
    |\lambda|\lesssim_{\kappa,\rho}
    \begin{cases}
    \delta^{-\frac{(d-1)^2}{2}}\|\e^{2\sqrt{1+\delta}(\im\lambda)_-|\cdot|}V\|_{(d+1)/2,1/2}^{(d+1)/2}\quad &\mbox{(Lorentz space)},\\
\delta^{-(d-1)}R \langle \lambda R\rangle^{\kappa}\e^{2\sqrt{1+\delta}(\im\lambda)_-R}\|V\|_{\infty}\quad &\mbox{(compactly supported)},\\
\delta^{-(d-1)}R\langle \lambda R\rangle^{\kappa} v_{\rho,R,\sqrt{1+\delta}(\im \lambda)_-}\quad &\mbox{(pointwise decaying)}.
    \end{cases}
\end{align}
\item [(b)]Moreover, there exists $A_0>0$, depending only on $d$, such that if $A\geq A_0$ and $s>1$,
then $n_V(r)\lesssim_{\nu,\kappa,\rho}(\ln s)^{-1} (n_+(sr)+n_-(sr)+n_0$ (the same upper bound holds for the number of resonances in the larger disk $D(\lambda_0,|\lambda_0|+r)$) with
\begin{align*}
 n_+(r):=&
 \begin{cases}
 r^{-2}\|V\|_{(d+1)/2}^{d+1},\\
 (\langle r R+AR\|V\|_{\infty}^{\frac{1}{2}}\rangle^{\nu+\kappa} Rr^{-1}\|V\|_{\infty})^{\frac{d-1}{\nu}}+(Rr)^d(r^{-1}\|V\|_{\infty}^{\frac{1}{2}})^\frac{d}{\theta},\\
(\langle r R+ARv_{\rho,R}^{\frac{1}{2}}\rangle^{\nu+\kappa} Rr^{-1}v_{\rho,R})^{\frac{d-1}{\nu}}+(Rr)^d(r^{-1}v_{\rho,R}^{\frac{1}{2}})^\frac{d}{\theta},
\end{cases}\\
 n_-(r):=&
 \begin{cases}
\epsilon^{1-d}[\ln(C\delta^{-\frac{(d-1)^2}{d+1}} r^{-\frac{2}{d+1}}\|\e^{2\gamma(sr)|\cdot|}V\|_{(d+1)/2,1/2})]^d,\\
\epsilon^{1-d}[\ln(C\delta^{1-d}Rr^{-1}\e^{2\gamma(sr) R}\|V\|_{\infty})]^d,\\
\epsilon^{1-d}[\ln (C\delta^{1-d}Rr^{-1}v_{\rho,R,\gamma(sr)})]^d,
\end{cases}\\
 n_0:=&
 \begin{cases}
 A^{-2},\\
 R^dA^{-\frac{d(1-\theta)}{\theta}}\|V\|_{\infty}^{\frac{d}{2}},\\
 R^dA^{-\frac{d(1-\theta)}{\theta}}v_{\rho,R}^{\frac{d}{2}},
\end{cases}\\
\gamma(r):=&
\begin{cases}
(1+\delta)^{\frac{1}{2}}r+\epsilon(A\|V\|_{(d+1)/2}^{(d+1)/2}r+r^2)^{\frac{1}{2}},\\
(1+\delta)^{\frac{1}{2}}r+\epsilon (A\|V\|_{\infty}^{\frac{1}{2}}r+r^2)^{\frac{1}{2}},\\
(1+\delta)^{\frac{1}{2}}r+\epsilon (Av_{\rho,R}^{\frac{1}{2}}r+r^2)^{\frac{1}{2}},
\end{cases}
\end{align*}
\end{enumerate}
\end{theorem}

\begin{remark}
        For $\im\lambda\geq 0$, the exponential weight and the negative power of $\delta$ in (a) can be omitted. Moreover, the (second) Lorentz exponent can be omitted. More precisely, we have (see Appendix \ref{appendix Details of the proof of Theorem 7.1} for the proof)
   \begin{align}\label{eigenvalue bounds, i.e. Im lambda>0}
    |\lambda|\lesssim_{\kappa,\rho}
    \begin{cases}
    \|V\|_{(d+1)/2}^{(d+1)/2}\quad &\mbox{(Lorentz space)},\\
R\langle \lambda R\rangle^{\kappa} \|V\|_{\infty}\quad &\mbox{(compactly supported)},\\
R\langle \lambda R\rangle^{\kappa} v_{\rho,R}\quad &\mbox{(pointwise decaying)},.
    \end{cases}
    \quad \im\lambda\geq 0.
\end{align}
\end{remark}

\begin{remark}\label{remark nu+kappa<1}
    For the compactly supported and pointwise decaying case, the term $\langle r R+AR\|V\|_{\infty}^{\frac{1}{2}}\rangle^{\nu+\kappa}$ in $n_+(r)$ can be replaced by $\langle r R\rangle^{\nu+\kappa}$ if $\nu+\kappa\leq 1$, see Appendix B.2.2 for details.
\end{remark}

\begin{proof}[Proof of Theorem \ref{thm. main intro} assuming Theorem \ref{thm. main detailed}]
    We only give details for the number of resonances; the bounds \eqref{main thm resonance free region Lp}, \eqref{main thm resonance free region cpt}, \eqref{main thm resonance free region ptw} are obvious. To avoid confusion, we rename $\epsilon$ in Theorem \ref{thm. main intro} to $\epsilon_1$. 
    
    Consider first the Lorentz case. Let $A=A_0$ in Theorem \ref{thm. main detailed}. The lower bound on $r$ in \eqref{main thm n_V Lp} then implies that 
    \begin{align}
        \gamma(sr)\leq(1+\delta)^{\frac{1}{2}}sr+\epsilon r(s^2+sA_0C_2^{-1})^{\frac{1}{2}}.
    \end{align}
    Hence, there exist $\delta=\delta(\epsilon_1)>0$, $\epsilon=\epsilon(\epsilon_1)>0$ and $s=s(\epsilon_1)>1$such that $\gamma(sr)\leq (1+\epsilon_1)r$. It only remains to observe that $n_+(r),n_0\lesssim n_-(r)$. This is clear since $n_+(r),n_0\lesssim 1$ (for $n_+(r)$ this again follows from the lower bound on $r$ in \eqref{main thm n_V Lp}). 
    
    In the compactly supported case we take $A=Cr\|V\|_{\infty}^{-1/2}$ with some constant $C\geq A_0/C_2$. Then again there exist $\delta>0,\epsilon>0,s>1$ such that $\gamma(sr)\leq (1+\epsilon_1)r$. Moreover, since $A\gtrsim 1$, we have $n_0(r)\lesssim n_+(r)$. In the following, we fix $\theta=1/2$ for definiteness, but any $\theta\in [1/2,1]$ would work. The lower bound on $r$ in \eqref{main thm n_V cpt} implies that $\|V\|_{\infty}\lesssim r^2$, so that
    \begin{align*}
        n_+(r)\lesssim_{\nu,\kappa} (\langle rR\rangle^{\nu+\kappa} Rr)^{\frac{d-1}{\nu}}+(Rr)^d\lesssim \langle Rr\rangle^d,
    \end{align*}
    where we fixed $\nu=d$ in the last inequality (any $\nu>d-1$ would work) and chose $\kappa$ small enough. If $Rr<1$, then clearly $n_+(r)\lesssim 1\lesssim n_-(r)$, where the second inequality holds in view of the final comment in Remark \ref{remark literature results}. If $Rr\geq 1$, we first observe that \eqref{main thm resonance free region cpt}, with $\epsilon_1/2$ in place of $\epsilon$, yields
    \begin{align*}
        1\leq C_1(\epsilon_1/2)\langle rR\rangle^{\epsilon} r^{-1}R\e^{2(1+\frac{\epsilon_1}{2})rR}\|V\|_{\infty}.
    \end{align*}
    Together with $\langle rR\rangle^{\epsilon}\lesssim_{\epsilon,\epsilon_1}\e^{\frac{\epsilon_1}{2}rR}$, this implies $n_-(r)\gtrsim_{\epsilon_1} (Rr)^d\gtrsim n_+(r)$.

 The proof in the pointwise decaying case is similar; here we select $\kappa>0$ and $\nu>d-1$ such that $1+\nu+\kappa\leq d+\epsilon_1$.
\end{proof}

Before giving the proof of Theorem \ref{thm. main detailed} it is useful to record the following elementary estimates.

\begin{lemma}\label{lemma sum ln with M}
i) For any $M>0$, $\alpha>\beta>0$, we have
\begin{align}\label{sum ln with M i)}
\sum_{k\in\N}\ln(1+(Mk^{-1/\beta})^{\alpha})\leq C_{\alpha,\beta} M^{\beta}.
\end{align}
ii) For any $M>2$, $\alpha,\beta,\epsilon>0$, we have
\begin{align}\label{sum ln with M ii)}
\sum_{k\in\N}\ln(1+(M \exp(-\epsilon k^{1/\beta}))^{\alpha})\leq C_{\alpha,\beta} \epsilon^{-\beta}[\ln M]^{\beta+1}.
\end{align}
\end{lemma}

\begin{proof}
Without loss of generality we may assume that $\alpha=1$, otherwise we replace $M,\beta$, by $M^{\alpha},\beta/\alpha$ in i) and $M,\epsilon$ by $M^{\alpha},\alpha\epsilon$ in ii). 

i) Set $\mu_k:=Mk^{-1/\beta}$.
The estimate \eqref{sum ln with M i)} is obvious if $M\leq 1$ since then
\begin{align*}
\sum_k\ln(1+\mu_k)
\leq M\sum_k k^{-1/\beta}
\lesssim_{\beta} M\leq M^{\beta},
\end{align*}
where we used that $\beta<1$.
If $M>1$, then we have
\begin{align*}
&\sum_{k>M^{\beta}}\ln(1+\mu_k)\leq \sum_{k>M^{\beta}}Mk^{-1/\beta}\lesssim M(M^{\beta})^{1-1/\beta}= M^{\beta},\\
&\sum_{k\leq M^{\beta}}\ln(1+\mu_k)\lesssim \ln(1+M)+\int_1^{M^{\beta}} \ln(1+Mx^{-1/\beta})\rd x\\
&\lesssim \ln(1+M)+M^{\beta}\int_1^{M} \ln(1+y)y^{-1-\beta}\rd y\lesssim_{\beta} M^{\beta}.
\end{align*}

ii) Set $\mu_k:=M\exp(-\epsilon k^{1/\beta})$. For any $K\in\N$ we have
\begin{align}\label{sum <K and >K}
\sum_k\ln(1+\mu_k)\leq \sum_{k\leq K}\ln(1+\mu_k)+\sum_{k>K}\mu_k,
\end{align}
where we again used $\ln(1+\mu_k)\leq \mu_k$ in the second sum. Estimating the latter by an integral, we have
\begin{align}\label{sum > K}
\sum_{k>K}\mu_k\leq M\int_K^{\infty}\exp(-\epsilon y^{\frac{1}{\beta}})\rd y\lesssim  M\epsilon^{-\beta}\exp\left(-\frac{1}{2}\epsilon K^{\frac{1}{\beta}}\right)
\end{align}
On the other hand, since $\mu_k$ is decreasing, 
\begin{align}\label{sum < K}
\sum_{k\leq K}\ln(1+\mu_k)\leq K\ln M.
\end{align}
Combining \eqref{sum <K and >K}, \eqref{sum > K}, \eqref{sum < K}, we have
\begin{align*}
    \sum_k \ln(1+\mu_k)\lesssim K\ln M+M\epsilon^{-\beta}\exp\left(-\frac{1}{2}\epsilon K^{\frac{1}{\beta}}\right)
\end{align*}
for any $K\in \N$. We choose $K^{\frac{1}{\beta}}=(2/\epsilon)\ln(CM)$, where $C>1$ is such that $K\in\N$. This yields
\begin{align}\label{sum final}
\sum_k \ln(1+\mu_k)\lesssim \epsilon^{-\beta}[\ln(CM)]^{\beta}\ln M+\epsilon^{-\beta}\frac{M}{1+CM}
\lesssim \epsilon^{-\beta}[\ln M]^{\beta+1}.
\end{align}
\end{proof}

\begin{proof}[Proof of Theorem \ref{thm. main detailed}]
We only provide the main arguments here, details are postponed to Appendix \ref{appendix Details of the proof of Theorem 7.1}.

(a) This a consequence of the Birman--Schwinger principle: If $\lambda$ is a resonance, then $\|BS(\lambda)\|\geq 1$. Using the bounds \eqref{singular values Lp im lm>0}, \eqref{singular values Lp im lm <0 (used for Schatten membership)} with $k=1$ yields an upper bound on $s_1(BS(\lambda))=\|BS(\lambda)\|$ in the Lorentz case. In the compactly supported case we use \eqref{cpt. singular values im lm >0}, \eqref{cpt. singular values im lm<0} (with $\theta=1/2$, $\nu=0$). In the pointwise decaying case we use \eqref{ptw. singular values im lm >0}, \eqref{ptw. singular values im lm <0 (needed for Schatten membership)} (with the same choice of $\theta,\nu$). 

(b) The starting point is \eqref{conclusion sect. 6-1}, \eqref{conclusion sect. 6-2}, i.e. the estimate on $n_V(r)$ in terms of the singular values of $BS(\lambda)$. We point out that the constant $A$ in the theorem is the same constant that appear in \eqref{def lambda0}.
In the Lorentz space case we take $\alpha=d+1$, in the other two cases we take $\alpha>\max(\frac{d-1}{\nu},\frac{d}{2\theta})$. The exact choice of $\alpha$ is not important, but the constants will depend on $\alpha$. Our choice guarantees that the sums in the last display converge and are bounded by the expressions in the theorem. This is a consequence of 
the singular value estimates in Section \ref{sect. singular value estimates} and Lemma \ref{lemma sum ln with M}. 
Indeed, the bounds for $n_+(r)$ follow immediately from \eqref{singular values Lp im lm>0}, \eqref{cpt. singular values im lm >0}, \eqref{ptw. singular values im lm >0} and \eqref{sum ln with M i)}, recalling that $|\lambda|\geq r$ for $\lambda\in S_+$ (in the compactly supported and pointwise decaying case, an additional argument is needed for $\nu+\kappa>1$, see Appendix \ref{appendix Details of the proof of Theorem 7.1}). Those for $n_0(r)$ follow from \eqref{singular values Lp im lm>0}, \eqref{cpt. singular values exterior of cone theta}, \eqref{ptw. singular values exterior of cone theta} and \eqref{sum ln with M i)} and the definition of $\lambda_0$, see \eqref{def lambda0}. Those for $n_-(r)$ follow from \eqref{singular values Lp exponential}, \eqref{cpt. singular values exponential}, \eqref{ptw. singular values exponential} and \eqref{sum ln with M ii)}, with
\begin{align*}M=C
    \begin{cases}
        |\lambda|^{-\frac{2}{d+1}}\delta^{-\frac{(d-1)^2}{d+1}} \|\e^{2\gamma(sr)|\cdot|}V\|_{(d+1)/2,1/2},\\
        \delta^{1-d}R|\lambda|^{-1}\|\rho_{\delta,\epsilon}^{-2}V\|_{\infty},\\
        \delta^{1-d}R|\lambda|^{-1}v_{\rho,R,\gamma(sr)},
    \end{cases}
\end{align*}
and $\beta=d-1$. From the upper bound on $|\lambda|$ in part (a) it follows that there is a $C>0$ such that $M>2$; we choose $C$ in such a way. We also use the fact that $\im\lambda\geq -sr$, $|\lambda|\leq \sqrt{2|\lambda_0|sr+(sr)^2}$ for $\lambda\in S_-(sr,\lambda_0)$, which implies that $\|\rho_{\delta,\epsilon}^{-2}V\|_{\infty}\leq \e^{2\gamma(sr)R}V\|_{\infty})$.
\end{proof}

\appendix

\section{Proof of the structure formula}\label{Appendix Proof of the structure formula}
We give a proof of \eqref{structure formula}, \eqref{al bound structure formula}.

\begin{proof}
1. Integration by parts yields
\begin{align*}
\mathcal{E}(\lambda)(I-\epsilon^2\Delta_S)^l g(x)=\int_S [(I-\epsilon^2\Delta_S)^l \e^{\I \lambda x\cdot\xi}]g(\xi)\rd S(\xi),
\end{align*}
which shows that
\begin{align}
\mathcal{E}(\lambda)(I-\epsilon^2\Delta_S)^l=\mathcal{E}_{a_l}(\lambda),\quad \mbox{with}\quad a_l(x,\xi)=\e^{-\I \lambda x\cdot\xi}(I-\epsilon^2\Delta_S)^l \e^{\I \lambda x\cdot\xi}.  
\end{align}
We will prove the bound \eqref{al bound structure formula} for $N=0$, the other cases being similar. Let $v\in\C^d$ and let $\phi_v(\xi):=v\cdot\xi/|\xi|$, $\xi\in \R^d\setminus\{0\}$, so that $\e^{\phi_v(\xi)}$ is the degree zero homogeneous extension of $\e^{\I \lambda x\cdot\xi}$ for  $v=\I\lambda x$. By a well-known formula for the spherical Laplacian,
\begin{align*}
\Delta_S \,\e^{\I \lambda x\cdot\xi}=\Delta \e^{\phi_v(\xi)},\quad \xi\in S,
\end{align*}
where $\Delta=\sum_{j=1}^{d}\partial_{\xi_j}$ is the Euclidean Laplacian. It thus suffices to prove that 
\begin{align}
    |\e^{-\phi_v(\xi)}(I-\epsilon^2\Delta)^l \e^{\phi_v(\xi)}|\leq C^{2l}(2l)!\exp(\epsilon v),\quad v\in\C^d,\quad \xi\in \R^d\setminus\{0\}.
\end{align}

2. By the multinomial formula, the left hand side of the last display consists of $(d+1)^l$ terms of the form $\e^{-\phi_v(\xi)}(\epsilon\partial)^{\alpha}\e^{\phi_v(\xi)}$ with $|\alpha|\leq 2l$. Note that the factor $(d+1)^l$ can be absorbed into the constant in \eqref{al bound structure formula}. Hence, we have reduced the task to proving that for $|\alpha|\leq 2l$, the function
\begin{align}\label{partial ephiv}
a_{\alpha}(v,\xi)=\e^{-\phi_v(\xi)}\partial^{\alpha} \e^{\phi_v(\xi)}
\end{align}
satisfies 
\begin{align}
    |\epsilon^{|\alpha|}a_{\alpha}(v,\xi)|\leq C^{|\alpha|}|\alpha|!\exp(\epsilon v),\quad v\in\C^d,\quad \xi\in \R^d\setminus\{0\}.
    \end{align}
Since this is a local statement, it suffices to prove it for some fixed $\xi\in S$. Since the upper bound in \eqref{al bound structure formula} is invariant under rotations, we may assume without loss of generality that $\xi=e_1$. 

3. It follows from the multivariate {F}a\`a di {B}runo formula (see \cite[Corollary 2.10]{MR1325915} for details) that
\begin{align}\label{Faa di Bruno formula}
a_{\alpha}(v,\xi)=\sum_{r=1}^{|\alpha|}\sum_{p(\alpha,r)}\alpha!\prod_{j=1}^{|\alpha|}\frac{[\partial^{\lambda_j}\phi_v(\xi)]^{k_j}}{(k_j!)(\lambda_j!)^{k_j}}
\end{align}
where $p(\alpha,r)$ is a set of $(k_1,\ldots,k_{|\alpha|};\lambda_1,\ldots,\lambda_{|\alpha|})$, $k_i\in\N_0$, $\lambda_i\in \N_0^{d}$, satisfying the constraints
\begin{align}\label{constraints Faa di Bruno formula}
\sum_{i=1}^{|\alpha|}k_i=r,\quad \sum_{i=1}^{|\alpha|}k_i\lambda_i=\alpha,
\end{align}
 Since $\phi_v$ is homogeneous of degree one in $v$ and since $\epsilon\leq 1$, it follows from \eqref{Faa di Bruno formula} and the first identity in \eqref{constraints Faa di Bruno formula} that, with $\omega:=v/|v|$, 
\begin{align*}
|\epsilon^{|\alpha|}a_{\alpha}(v,\xi)|&=\left|\sum_{r=1}^{|\alpha|}|\epsilon v|^r\sum_{p(\alpha,r)}\alpha!\prod_{j=1}^{|\alpha|}\frac{[\epsilon\partial^{\lambda_j}\phi_{\omega}(\xi)]^{k_j}}{(k_j!)(\lambda_j!)^{k_j}}\right|\\
&\leq \exp(|\epsilon v|)\sum_{p(\alpha,r)}r!\alpha!\prod_{j=1}^{|\alpha|}\frac{|\partial^{\lambda_j}\phi_{\omega}(\xi)|^{k_j}}{(k_j!)(\lambda_j!)^{k_j}}.
\end{align*}

4. It remains to show that
\begin{align}\label{combinatorial bound aalpha}
\sum_{p(\alpha,r)}r!\alpha!\prod_{j=1}^{|\alpha|}\frac{|\partial^{\lambda_j}\phi_{\omega}(\xi)|^{k_j}}{(k_j!)(\lambda_j!)^{k_j}}\leq C^{|\alpha|+1}\alpha!.
\end{align}
Here and in the following, $C$ denotes a generic constant (only depending on the dimension $d$) that is allows to change from line to line. 
To prove this we use that $\phi_{\omega}(\xi)$ is real-analytic at $\xi=e_1$, which means that we have derivative bounds
\begin{align}\label{real-analyticity of phi}
|\partial_{\lambda}\phi_{\omega}(\xi)|\leq C^{|\lambda|+1}\lambda!
\end{align} 
for $\xi$ in an open ball around $e_1$. Since $\omega$ ranges over the unit sphere, the constant can be taken uniform in $\omega$. Using \eqref{real-analyticity of phi} in \eqref{combinatorial bound aalpha} together with the second identity in \eqref{constraints Faa di Bruno formula} yields
\begin{align*}
\sum_{p(\alpha,r)}r!\alpha!\prod_{j=1}^{|\alpha|}\frac{|\partial^{\lambda_j}\phi_{\omega}(\xi)|^{k_j}}{(k_j!)(\lambda_j!)^{k_j}}\leq C^{|\alpha|+1}\alpha!r!\sum_{p(\alpha,r)}\frac{1}{\prod_{j=1}^{|\alpha|}k_j!}.
\end{align*}
The last term can be evaluated to (see \cite[page 515]{MR1325915})
\begin{align*}
r!\sum_{p(\alpha,r)}\frac{1}{\prod_{j=1}^{|\alpha|}k_j!}=\begin{pmatrix}
|\alpha|-1\\ r-1
\end{pmatrix},
\end{align*}
and we estimate
\begin{align*}
\begin{pmatrix}
|\alpha|-1\\ r-1
\end{pmatrix}
\leq \sum_{r=2}^{|\alpha|}\begin{pmatrix}
|\alpha|-1\\ r-1
\end{pmatrix}=2^{|\alpha|-1}.
\end{align*}
This proves \eqref{combinatorial bound aalpha}.
\end{proof}

\section{Details of the proof of Theorem 7.1}\label{appendix Details of the proof of Theorem 7.1}

\subsection{Details for part (a)}
We fist recall the Birman--Schwinger principle: If $\im\lambda\in\C$ is a resonance of $-\Delta+V$, then
\begin{align*}
    1\leq \|BS(\lambda)\|=s_1(BS(\lambda)).
\end{align*}
\subsubsection{Lorentz case:} We start with the easier bound \eqref{eigenvalue bounds, i.e. Im lambda>0} for $\im\lambda\geq 0$. By \eqref{singular values Lp im lm>0} with $k=1$, we have
\begin{align*}
    1\leq \|BS(\lambda)\|\lesssim |\lambda|^{-\frac{2}{d+1}}\|V\|_{(d+1)/2}\implies |\lambda|\lesssim \|V\|_{(d+1)/2}^{(d+1)/2}
\end{align*}
For $\im\lambda< 0$, we use \eqref{singular values Lp im lm <0 (used for Schatten membership)} with $k=1$,
\begin{align*}
 \|BS(\lambda)\|\lesssim \delta^{-\frac{(d-1)^2}{d+1}}  |\lambda|^{-\frac{2}{d+1}}\|\e^{2\sqrt{1+\delta}|\im\lambda||\cdot|}V\|_{(d+1)/2,1/2},
\end{align*}
which leads to the stated bound in Theorem \ref{thm. main detailed} (a) in the Lorentz case. Note that the bound there is stated for all $\lambda\in\C$, which is still true since the bound for $\im\lambda\geq 0$ is clearly dominated by that for $\im\lambda<0$. This remark also applies to the compactly supported and pointwise decaying case, and we will not repeat it.

\subsubsection{Compactly supported case:} For $\im\lambda\geq 0$, we use \eqref{cpt. singular values im lm >0} with $k=1$:
\begin{align*}
    \|BS(\lambda)\|\lesssim (\langle \lambda R\rangle^{\nu}\ln \langle \lambda R\rangle R|\lambda|^{-1}+R^{2\theta}|\lambda|^{-2(1-\theta)})\|V\|_{\infty}.
\end{align*}
Choosing $\theta=\frac{1}{2},\nu=0$, the first term dominates the second, and we get
\begin{align}\label{Appendix details compact Im lm>0}
    \|BS(\lambda)\|\lesssim\ln \langle \lambda R\rangle R|\lambda|^{-1}\|V\|_{\infty}.
\end{align}
Since $\ln \langle \lambda R\rangle\lesssim_{\kappa}\langle \lambda R\rangle^{\kappa}$ for any $\kappa>0$, we get the stated bound in \eqref{eigenvalue bounds, i.e. Im lambda>0}.

For $\im\lambda< 0$, we use \eqref{cpt. singular values im lm<0} with $k=1$:
\begin{align*}
    \|BS(\lambda)-BS(-\lambda)\|\lesssim \delta^{1-d}  \langle \lambda R\rangle^{\nu}\ln \langle \lambda R\rangle R|\lambda|^{-1}\|\rho_{\delta}^{-2}V\|_{\infty}.
\end{align*}
Since the upper bound \eqref{Appendix details compact Im lm>0} for $\im\lambda\geq 0$ is dominated by the right hand side in the last display, an application of the triangle inequality for the operator norm $\|\cdot\|$ yields the same upper bound for $\|BS(\lambda)\|$. This gives the stated bound in Theorem \ref{thm. main detailed} (a) in the compactly supported case.

\subsubsection{Pointwise decaying case:} For $\im\lambda\geq 0$, we use \eqref{ptw. singular values im lm >0} with $k=1$:
\begin{align}
\|BS(\lambda)\|\lesssim (\langle \lambda R\rangle^{\nu+\kappa}R|\lambda|^{-1}+R^{2\theta}|\lambda|^{-2(1-\theta)})v_{\rho,R}.
\end{align}
Again, choosing $\theta=\frac{1}{2},\nu=0$ and observing that the first term dominates the second, we get
\begin{align}
    \|BS(\lambda)\|\lesssim \langle \lambda R\rangle^{\kappa} R|\lambda|^{-1}v_{\rho,R},
\end{align}
which gives the stated bound in \eqref{eigenvalue bounds, i.e. Im lambda>0}.

For $\im\lambda< 0$, we use \eqref{ptw. singular values im lm <0 (needed for Schatten membership)} with $k=1$:
\begin{align}
\|BS(\lambda)\|
\lesssim \delta^{1-d} \langle \lambda R\rangle^{\nu+\kappa}R|\lambda|^{-1}v_{\rho,R,\sqrt{1+\delta}|\im\lambda|}+R^{2\theta}|\lambda|^{-2(1-\theta)}v_{\rho,R},
\end{align}
Choosing $\theta=\frac{1}{2},\nu=0$, the first term again dominates the second, and we get
\begin{align*}
   \|BS(\lambda)\|
\lesssim \delta^{1-d} \langle \lambda R\rangle^{\kappa}R|\lambda|^{-1}v_{\rho,R,\sqrt{1+\delta}|\im\lambda|} 
\end{align*}
This gives the stated bound in Theorem \ref{thm. main detailed} (a) in the pointwise decaying case.

\subsection{Details for part (b)}
Recall \eqref{conclusion sect. 6-1}, \eqref{conclusion sect. 6-2}:
\begin{align}
    n_V(r)\lesssim (\ln s)^{-1}(\tilde{n}_+(sr)+\tilde{n}_-(sr)+\tilde{n}_0),
\end{align}
\begin{align}
        \tilde{n}_{\pm}(r):=\max_{\lambda\in S_{\pm}(r,\lambda_0)} \sum_{k\in\N}\ln(1+[s_k(BS(\lambda))]^{\alpha}),\quad
         \tilde{n}_0:=\sum_{k\in\N}\ln(1+[s_k(BS(\lambda_0))]^{\alpha}),
\end{align}
where we omitted the dependence of the implicit constant on the (fixed) integer $\alpha>d/2$. We will show that, if $n_{\pm}(r)$, $n_0$ denote the quantities appearing in the statement of Theorem \ref{thm. main detailed}, then
\begin{align}
   \tilde{n}_{+}(r)\lesssim n_+(r), \quad \tilde{n}_{-}(r)\lesssim n_+(r)+n_-(r),\quad \tilde{n_0}\lesssim n_0.
\end{align}
\subsubsection{Lorentz case:} Using \eqref{singular values Lp im lm>0} and the fact that $x\mapsto \ln(1+x^{\alpha})$ is monotonically increasing on $(0,\infty)$, we get
\begin{align}
 \tilde{n}_{+}(r)\lesssim \max_{\lambda\in S_{+}(r,\lambda_0)}\sum_k \ln(1+[k^{-\frac{1}{d+1}} |\lambda|^{-\frac{2}{d+1}}\|V\|_{(d+1)/2}]^{\alpha}). 
\end{align}
Lemma \ref{lemma sum ln with M} i) applied with $M=|\lambda|^{-\frac{2}{d+1}}\|V\|_{(d+1)/2}$ and $\beta=d+1$ yields
\begin{align*}
 \tilde{n}_{+}(r)\lesssim \max_{\lambda\in S_{+}(r,\lambda_0)} |\lambda|^{-2}\|V\|_{(d+1)/2}^{(d+1)/2}\leq r^{-2}\|V\|_{(d+1)/2}^{(d+1)/2}=n_+(r),
\end{align*}
where, in the last inequality, we used that $|\lambda|\geq r$ for $\lambda\in S_{+}(r,\lambda_0)$.

To estimate $\tilde{n}_{+}(r)$ we first sparsify the sum over $k\in\N$, by splitting it into a sum over even and odd $k$ (a similar argument was used in Section 6). By monotonicity of the singular values, it suffices to consider one of these, say the even ones. Using $s_{2j+1}(A+B)\leq s_{j+1}(A)+s_{j+1}(B)$, which follows from \eqref{s_j(AB)}, we get
\begin{align}\label{first sum}
 \sum_{k\in 2\N}\ln(1+[s_k(BS(\lambda))]^{\alpha})\lesssim \sum_{k\in N}\ln(1+[2s_k(BS(-\lambda))]^{\alpha})\\+\sum_{k\in N}\ln(1+[2s_k((BS(\lambda)-BS(-\lambda)))]^{\alpha}).
\end{align}
Here we again used the monotonicity of $x\mapsto \ln(1+x^{\alpha})$, as well as its consequence
\begin{align*}
\ln(1+(x+y)^{\alpha})&\leq \ln(1+2^{\alpha}x^{\alpha}+2^{\alpha}y^{\alpha})
\leq \ln((1+2^{\alpha}x^{\alpha})(1+2^{\alpha}y^{\alpha}))\\
&=\ln(1+2^{\alpha}x^{\alpha})+\ln(1+2^{\alpha}y^{\alpha}).
\end{align*}
Taking the maximum over $\lambda\in S_{+}(r,\lambda_0)$, the first sum on the right hand side of \eqref{first sum} is estimated by $n_+(r)$, so it remains to show that the second sum is dominated by $n_-(r)$. By
\eqref{singular values Lp exponential}, for $\lambda\in S_{+}(r,\lambda_0)$, this sum is bounded by
\begin{align}
\sum_{k\in N}\ln(1+[2^{\alpha}|\lambda|^{-\frac{2}{d+1}}\delta^{-\frac{(d-1)^2}{d+1}}\exp\left(-c\epsilon k^{\frac{1}{d-1}}\right)\|\e^{2\gamma(r)|\cdot|}V\|_{(d+1)/2,1/2}]^{\alpha}).
\end{align}
Lemma \ref{lemma sum ln with M} ii), with $M=2^{\alpha}|\lambda|^{-\frac{2}{d+1}}\delta^{-\frac{(d-1)^2}{d+1}}\|\e^{2\gamma(r)|\cdot|}V\|_{(d+1)/2,1/2}$ and $\beta=d-1$, shows that this is dominated by $n_-(r)$.

The bound for $\tilde{n}_0$ is similar to that for $\tilde{n}_+$. We again use \eqref{singular values Lp im lm>0}, but this time with $\lambda_0$ instead of $\lambda$, resulting in
\begin{align}
 \tilde{n}_0\lesssim\sum_k \ln(1+[k^{-\frac{1}{d+1}} (A/2))^{-\frac{2}{d+1}}]^{\alpha}), 
\end{align}
where we used the definition of $\lambda_0$, see \eqref{def lambda0}.
Lemma \ref{lemma sum ln with M} i) applied with $M=(A/2))^{-\frac{2}{d+1}}\|V\|_{(d+1)/2}$ and $\beta=d+1$ yields $\tilde{n}_0\lesssim A^{-2}=n_0$.

\subsubsection{Compactly supported and decaying case}
The argument is exactly the same as for the Lorentz case, with the appropriate replacement of the singular value bounds (the relevant bounds are cited in the proof of Theorem \ref{thm. main detailed}). The only difference is the appearance of a factor $\langle \lambda R\rangle^{\nu+\kappa}$ with a \emph{positive} power of $\lambda$. But for $\lambda\in S_{\pm}(r,\lambda_0)$, we have a prior only a lower bound $|\lambda|\geq r$, so it is not immediately clear how to estimate the maximum over such $\lambda$. This factor appears in the expression $\langle \lambda R\rangle^{\nu+\kappa}R|\lambda|^{-1}\|V\|_{\infty}$ (or $v_{\rho,R}$). If $\nu+\kappa\leq 1$, then we can estimate
\begin{align*}
    \langle \lambda R\rangle^{\nu+\kappa}R|\lambda|^{-1}
    \lesssim \langle rR\rangle^{\nu+\kappa}Rr^{-1}
\end{align*}
for $|\lambda|\geq r$. Indeed, if $|\lambda|R\leq 2$, then this clearly holds. If $|\lambda|R>2$, then the left hand side is bounded by $|\lambda|^{\nu+\kappa-1} R^{\nu+\kappa+1}$. Since $\nu+\kappa\leq 1$, the claim follows. However, if $\nu+\kappa> 1$, then the lower bound $|\lambda|\geq r$ is not enough, and we need to use the stronger fact that $\lambda$ lies in the intersection of the upper half plane with $\partial D(\lambda_0,|\lambda_0|+r)$, which implies that $|\lambda|\leq 2|\lambda_0|+r$ (see Figure 1 in Section \ref{sect. Fredholm det.}). This leads to the bound for $n_+(r)$ stated in Theorem \ref{thm. main detailed}.

\section{Proof of Remark \ref{remark literature results} and Example \ref{example sparse}}\label{appendix examples}
\subsection{Proof of Remark \ref{remark literature results}}
(i) Assume that $V\in L^{\infty}_{\rm comp}(\R^d)$, $\supp(V)\subset B(0,R)$. It suffices to show that the right hand sides of \eqref{main thm n_V Lp}, \eqref{main thm n_V cpt}, \eqref{main thm n_V ptw} are all bounded by $C_V r^d$ for $r\geq 1$. This is obvious for \eqref{main thm n_V cpt}, \eqref{main thm n_V ptw}. 
To estimate the Lorentz norm in \eqref{main thm n_V Lp} we use 
\begin{align*}
|\{x:e^{2\gamma|x|}|V(x)|>\alpha\}|\leq |B(0,R)|\mathbf{1}\{\alpha<\e^{2\gamma R}\|V\|_{\infty}\}
\end{align*}
to find that (with $\gamma=1+\epsilon$)
\begin{align*}
\|\e^{2(1+\epsilon)r|\cdot|}V\|_{(d+1)/2,1/2}\lesssim \left(\int_0^{\infty}\alpha^{-\frac{1}{2}}|\{x:e^{2\gamma|x|}|V(x)|>\alpha\}|^{\frac{1}{d+1}}\rd\alpha\right)^{2}\lesssim R^{\frac{2d}{d+1}}e^{2\gamma R}\|V\|_{\infty}.
\end{align*}

(ii) Without loss of generality we assume that $|V(x)|\leq \exp(-|x|^{1+\epsilon})$. We show that the right hand sides of \eqref{main thm n_V Lp}, \eqref{main thm n_V ptw} are bounded by $C_Vr^{d(1+1/\epsilon)}$ for $r\geq 1$.
In the pointwise decaying case \eqref{main thm n_V ptw} this follows from the fact that
\begin{align}\label{example 2 claim}
|V(x)|\leq C_{\rho,\epsilon}(1+|x|)^{-\rho}\exp(4^{1/\epsilon}\gamma^{1+1/\epsilon})\e^{-2\gamma|x|}
\end{align}
for every $\rho,\gamma>0$. Indeed, if $|x|^{\epsilon}\geq 4\gamma$, then 
\begin{align*}
|V(x)|\leq \e^{-\frac{1}{2}|x|^{1+\epsilon}}\e^{-2\gamma|x|}\leq C_{\rho,\epsilon}(1+|x|)^{-\rho} \e^{-2\gamma|x|},
\end{align*}
and if $|x|^{\epsilon}< 4\gamma$, then $\e^{-\gamma|x|}>\exp(-4^{1/\epsilon}\gamma^{1+1/\epsilon})$, so that 
\begin{align*}
|V(x)|\leq C_{\rho,\epsilon}(1+|x|)^{-\rho} <C_{\rho,\epsilon}(1+|x|)^{-\rho}\exp(4^{1/\epsilon}\gamma^{1+1/\epsilon})\e^{-2\gamma|x|}.
\end{align*}
In the Lorentz case \eqref{main thm n_V Lp}, let
\begin{align*}
U_{\alpha}:=\{x:e^{2\gamma|x|}|V(x)|>\alpha\},\quad U_{\alpha,<}:=U_{\alpha}\cap B(0,(4\gamma)^{\frac{1}{\epsilon}}).
\end{align*}
Then for $x\in U_{\alpha}\setminus U_{\alpha,<}$, we have $2\gamma|x|-|x|^{1+\epsilon}\leq -\frac{1}{2}|x|^{1+\epsilon}$, which implies
\begin{align*}
| U_{\alpha}\setminus U_{\alpha,<}|\lesssim [\ln(1/\alpha)]^{\frac{d}{1+\epsilon}}\mathbf{1}_{\alpha<1}.
\end{align*}
On the other hand, we have
\begin{align*}
|U_{\alpha,<}|\lesssim \gamma^{\frac{d}{\epsilon}}\mathbf{1}_{\alpha<\alpha_{\rm max}},
\end{align*}
where $\alpha_{\rm max}=\max_{x\in\R^d}e^{2\gamma|x|}|V(x)|\leq \exp(C_{\epsilon}\gamma^{1+1/\epsilon})$, in view of \eqref{example 2 claim}. Combining the last two displays, we have
\begin{align*}
&\|\e^{2(1+\epsilon)r|\cdot|}V\|_{(d+1)/2,1/2}\lesssim\left(\int_0^{\infty}\alpha^{-\frac{1}{2}}|U_{\alpha}|^{\frac{1}{d+1}}\rd\alpha\right)^{2}\\
&\lesssim \left(\int_0^1 \alpha^{-\frac{1}{2}}[\ln(1/\alpha)]^{\frac{d}{(1+\epsilon)(d+1)}}\rd\alpha\right)^{2}+\gamma^{\frac{d}{\epsilon}}\alpha_{\rm max}\leq \exp(C_{\epsilon}\gamma^{1+1/\epsilon}),
\end{align*}
where we used that the integral involving the logarithm is $\mathcal{O}(1)$. 

(iii) Assume again that $V\in L^{\infty}_{\rm comp}(\R^d)$, $\supp(V)\subset B(0,R)$. Without loss of generality we may assume that $R>2$ and $\|V\|_{\infty}\leq C_2^{-2}$. Since $n_V(r,h)=n_{V/h^2}(r/h)$, it is easy to see that \eqref{main thm n_V cpt} implies $n_V(r,h)\lesssim [rh^{-1}R]^d$ for $r\geq 1$, $h\in (0,1]$. 
The same argument works for \eqref{main thm n_V ptw}. It does not work for \eqref{main thm n_V Lp} since the condition $r\geq C_2 \|V\|_{(d+1)/2}^{(d+1)/2}$ is not invariant under rescaling $r\to r/h$, $V\to V/h^2$.

\subsection{Details of Example \ref{example sparse}}
Since $V$ is sparse, we have $L_j(1+o(1))\geq |x|$ for $x\in\Omega_j$ \cite[(16)]{MR4426735}.
Therefore, $$\|\e^{2\gamma|\cdot|}V\|_{(d+1)/2,1/2}\leq\|\sum_{j}\e^{\mathcal{O}(1)\gamma L_j}|H_j|\mathbf{1}_{\Omega_j}\|_{(d+1)/2,1/2}.$$ We assume $|H_j|=1$, $L_j=j$, $|\Omega_j|\leq C_M\exp(-Mj)$ for all $M>0$. Then
\begin{align*}
|\{x:\sum_{j}\e^{\mathcal{O}(1)\gamma L_j}|H_j|\mathbf{1}_{\Omega_j}(x)>\alpha\}|\leq \sum_{j>c\ln\alpha/\gamma}|\Omega_j|\lesssim C_M\alpha^{-cM/\gamma}
\end{align*}
for all $\alpha>1$, some $c>0$, and zero else. Thus,
\begin{align*}
\|\e^{2\gamma|\cdot|}V\|_{(d+1)/2,1/2}\lesssim \int_1^{\infty}(C_M\alpha^{-cM/\gamma})^{\frac{2}{d+1}}\rd\alpha
\lesssim C_M^{\frac{2}{d+1}}(c'M/\gamma-1)^{-1}
\end{align*}
for $\gamma<c'M$ and some $c'>0$.
Thus, \eqref{main thm n_V Lp} yields 
\begin{align*}
    n_V(r)\lesssim \left[\ln\left(\mathcal{O}(1)(C_M/r)^{\frac{2}{d+1}}\frac{r}{M-\mathcal{O}(1)r)}\right)\right]^d
\end{align*}
for $1\lesssim r\ll M$. Since $M$ can be taken arbitrarily large, it follows that $n_V(r)<\infty$ for all $r>0$.

\bibliographystyle{alpha}

\end{document}